\documentclass[10pt,reqno]{amsart}

\usepackage{amsmath}
\usepackage{color}
\usepackage[dvipsnames]{xcolor}
\usepackage{caption}  
\usepackage{makecell}
\usepackage{float}
\restylefloat{table}

\usepackage{tablefootnote}
\usepackage{mathrsfs}
\usepackage{amssymb}
\usepackage{amsthm}
\usepackage{mathtools}
\usepackage{mathabx}
\usepackage{url}

\usepackage{enumitem}
\usepackage{bbm}
\usepackage{times}
\usepackage[colorinlistoftodos]{todonotes}
\usepackage{csquotes}
\usepackage{cancel}
\usepackage{subcaption}
\usepackage{lipsum}
\usepackage[colorlinks=true,linkcolor=blue,citecolor=blue,linktocpage=true,urlcolor=black]{hyperref} 
\usepackage{marginnote}

\usepackage[capitalize]{cleveref}
\usepackage[toc, title]{appendix}
\usepackage{caption}
\usepackage[labelformat=simple]{subcaption}

\usepackage{setspace}     
\usepackage{centernot,cancel}

\allowdisplaybreaks

\newcommand{\Gl}{\mathrm{GL}}

\newcommand{\PP}{\mathbb{P}}

\newcommand{\RR}{\mathbb{R}}

\newcommand{\TT}{\mathbb{T}}

\newcommand{\ZZ}{\mathbb{Z}}

\newcommand{\cC}{\mathcal{C}}
\newcommand{\cD}{\mathcal{D}}
\newcommand{\cE}{\mathcal{E}}
\newcommand{\cF}{\mathcal{F}}

\newcommand{\cO}{\mathcal{O}}

\newcommand{\cU}{\mathcal{U}}

\newcommand{\Bcal}{\mathcal{B}}

\newcommand{\Nbb}{\mathbb{N}}
\newcommand{\Zbb}{\mathbb{Z}}
\newcommand{\Rbb}{\mathbb{R}}

\newcommand{\Leb}{\mathrm{Leb}}
\newcommand{\pr}{\mathrm{pr}}
\newcommand{\supp}{\mathrm{supp}}
\newcommand{\Prob}{\mathrm{Prob}}
\newcommand{\Lip}{\mathrm{Lip}}

\newcommand{\CS}[2]{\hypertarget{#1#2}{#1_{#2}}}
\newcommand{\C}[2]{\hyperlink{#1#2}{#1_{#2}}}

\newcommand{\OLD}[1]{}

\def\be#1\ee{\begin{align}\begin{split} #1 \end{split}\end{align}}
\def\beq#1\eeq{\begin{align*}\begin{split} #1 \end{split}\end{align*}}

\newcommand{\Sl}{\rm SL}

\newlist{enumlemma}{enumerate}{3} 
\setlist[enumlemma]{label*={ (\alph*)}, ref= {(\alph*)} }
\newlist{enumcount}{enumerate}{3} 
\setlist[enumcount]{label*={ (\arabic*)}, ref= {(\arabic*)} }

\DeclareMathOperator{\1}{\mathbf 1}

\long\def\symbolfootnote[#1]#2{\begingroup\def\thefootnote{\fnsymbol{footnote}}
\footnote[#1]{#2}\endgroup}

\DeclarePairedDelimiterX{\inner}[2]{\langle}{\rangle}{#1, #2}

\DeclareFontFamily{U}{wncy}{}
\DeclareFontShape{U}{wncy}{m}{n}{<->wncyr10}{}
\DeclareSymbolFont{mcy}{U}{wncy}{m}{n}
\DeclareMathSymbol{\Sh}{\mathord}{mcy}{"58}

\renewcommand{\phi}{\varphi}

\long\def\symbolfootnote[#1]#2{\begingroup\def\thefootnote{\fnsymbol{footnote}}
\footnote[#1]{#2}\endgroup}

\newcommand{\FFC}{{\mathfrak{C}}}

\newcommand{\SL}{{\rm SL}}

\newcommand{\Diff}{{\rm Diff}}

\newcommand{\Ucal}{ {\mathcal U}}

\newcommand{\Tbb}{{\mathbb T}}

\newcommand\restrict[1]{\raisebox{-.0ex}{$\upharpoonright$}_{#1}}

		\theoremstyle{Theorem}
\newtheorem{theorem}{Theorem} [section]

\newtheorem{ltheorem}{Theorem}

	\newtheorem{prop}[theorem]{Proposition} 
\newtheorem{claim}[theorem]{Claim}

\newtheorem{cor}[theorem]{Corollary}

\newtheorem{lem}[theorem]{Lemma}
\newtheorem*{theorem*}{Theorem}

\newtheorem{eg}[theorem]{Example}

		\theoremstyle{definition}

\newtheorem{Remark}[theorem]{Remark}

	\newtheorem{dfn}[theorem]{Definition}

	\newtheorem*{Not*}{Notation}

	\numberwithin{equation}{section}
	\numberwithin{theorem}{section}
	\makeindex

		\theoremstyle{remark}

\newtheorem{rmk}[theorem]{\textbf{Remark}}

		\theoremstyle{remark}
  \theoremstyle{theorem}

\begin{document}

\title{Absolute continuity of stationary measures}
\author{Aaron Brown}
\address{Northwestern University, 2033 Sheridan Road, Evanston, IL 60208}
\email{\href{mailto:awb@northwestern.edu}{awb@northwestern.edu}}

\author{Homin Lee}
\address{Northwestern University, 2033 Sheridan Road, Evanston, IL 60208}
\email{\href{mailto:homin.lee@northwestern.edu}{homin.lee@northwestern.edu}}
\author{Davi Obata}
\address{Brigham Young University,  275 TMCB Brigham Young University Provo, UT 84602 }
\email{\href{mailto:davi.obata@mathematics.byu.edu}{davi.obata@mathematics.byu.edu}}

\author{Yuping Ruan}
\address{Northwestern University, 2033 Sheridan Road, Evanston, IL 60208}
\email{\href{mailto:ruanyp@northwestern.edu; ruanyp@umich.edu}{ruanyp@northwestern.edu; ruanyp@umich.edu}}

\maketitle


\begin{abstract}
Let $f$ and $g$ be two volume preserving, Anosov diffeomorphisms on $\mathbb{T}^2$, sharing  common stable and unstable cones. In this paper, we find conditions for the existence of (dissipative) neighborhoods of $f$ and $g$,  $\mathcal{U}_f$ and $\mathcal{U}_g$, with the following property: for any probability measure $\mu$, supported on the union of these neighborhoods, and verifying certain conditions, the unique $\mu$-stationary SRB measure is absolutely continuous with respect to the ambient Haar measure.  Our proof is inspired in the work of Tsujii for partially hyperbolic endomorphisms \cite{Tsujii-bigpaper}. We also obtain some equidistribution results using the main result of \cite{Brown-Hertz}. 
\end{abstract}

\let\oldtocsection=\tocsection
\let\oldtocsubsection=\tocsubsection 
\let\oldtocsubsubsection=\tocsubsubsection
 
\renewcommand{\tocsection}[2]{\hspace{0em}\oldtocsection{#1}{#2}}
\renewcommand{\tocsubsection}[2]{\hspace{1em}\oldtocsubsection{#1}{#2}}
\renewcommand{\tocsubsubsection}[2]{\hspace{2em}\oldtocsubsubsection{#1}{#2}}
\setcounter{tocdepth}{2}

\tableofcontents

\section{Introduction}
Given a smooth action of a group $\Gamma$ on a manifold $M$, many natural questions arise including the extent to which it is possible to classify all orbit closures and all invariant or stationary measures.  For many homogeneous actions, the classification of orbit closures is very related to various number-theoretic questions.   

\subsection{Priori results in homogeneous, Teichm\"uller, and smooth dynamics}
As a motivating result, we recall a simple case of the main result of the seminal work  Benoist and Quint, \cite{MR2831114}.  As formulated, this also follows from the main result in the work by Bourgain, Furman, Lindenstrauss, and Mozes, \cite{MR2340439}.   Let $S=\{A_1, \cdots, A_k\}\in \Gl(n,\Zbb)$ and let $\Gamma$ denote the sub-semi-group generated by $S$.  We view each $A\in S$ and thus $\Gamma$ as acting on the torus $\mathbb{T}^n$ by automorphisms.  Given a probability measure $\mu$ on $S$, we say a probability measure $\nu$ on $\mathbb{T}^n$ is $\mu$-stationary if $\int A_* \nu \ d \mu(A)= \nu$.  
Assuming that (1) $\mu(A_i)>0$ for every $A_i\in S$ and (2) $\Gamma$, the semigroup generated by the support of $\mu$, is Zariski dense in  $\SL(n,\Rbb)$, in \cite{MR2831114} (see also \cite{MR2340439}) it is shown that:
\begin{enumerate}
\item every $\mu$-stationary probability measure $\nu$ on $\mathbb{T}^n$ is $\Gamma$-invariant;
\item every $\Gamma$-invariant probability measure $\nu$ on $\mathbb{T}^n$ is either finitely supported or the Haar measure; 
\item every $\Gamma$-orbit in $\mathbb{T}^n$ is either finite or dense.  
\end{enumerate}
Similar results for actions on semisimple homogeneos spaces $H/\Lambda$ and when the Zariski closure of $\Gamma$ is semisimple are obtained in \cite{MR2831114,MR3037785}.  

In the setting Teichm\"uller dynamics,   the (affine) action of $\Sl(2, \Rbb)$  on a strata $\mathcal H(\kappa)$ in the moduli space of  abelian differentials  on a surface was studied in the breakthrough work by Eskin and Mirzakhani in \cite{MR3814652}.  For the action of the upper-triangular subgroup $P\subset \Sl(2,\Rbb)$ and for certain measures  $\nu$ on   $\Sl(2,\Rbb)$, the  $P$-invariant and $\nu$-stationary measures are shown in \cite{MR3814652} to be  $\Sl(2,\Rbb)$-invariant and to coincide with  natural volume forms on affine submanifolds.   The classification of $P$-invariant measures was used in the work of Eskin, Mirzakhani, and Mohammadi (\cite{MR3418528}) to show that    $P$- and $\Sl(2,\Rbb)$-orbit closures are affine submanifolds.    

Beyond homogeneous or affine dynamics, for smooth ($C^2$ or $C^\infty$) actions on a manifold $M$ generated by finitely many diffeomoprhisms $\{f_1, \dots, f_k\}$, one would like a criterion on $\Gamma = \langle f_1, \dots, f_k \rangle$ that ensures a classification of stationary and invariant measures and of orbit closures.  For $C^2$-actions on surfaces,  \cite{Brown-Hertz}, the first author of this paper and Rodriguez Hertz provided a mechanism to classify all ergodic stationary measures satisfying a certain dynamical criterion (hyperbolicity and non-deterministicy of the associated Lyapunov flag) 
as either (1) finitely supported or (2)  satisfying the SRB property.   Such a classification is particularly useful when the generators $\{f_i\}$ are assumed to be volume preserving; in this case, all ergodic stationary measures satisfying the dynamical criterion are either finitely supported or an ergodic component of the ambient volume.  

One checkable criterion on a volume-preserving action that implies the dynamical criterion of \cite{Brown-Hertz} holds for every stationary measure is the uniform expansion criterion (see Section \ref{sec:equidistribution}). Under this criterion, in \cite{chung},  Chung used the classification in \cite{Brown-Hertz}  to classify all orbit closures for any volume-preserving, uniformly expanding $C^2$ action on a connected surface by showing all orbits are either finite or dense.  

\subsection{Overview of new results}
This paper continues the study of smooth ($C^2$) actions on surfaces.  
One question left unresolved in  \cite{Brown-Hertz} in the setting of dissipative group actions is the question of when a stationary measure satisfying the SRB property is absolutely continuous with respect to an ambient volume.  

Our main result in this paper provides a large class of group actions on the 2-torus $\mathbb{T}^2$ for which every ergodic stationary measure is either finitely supported or aboslutely continuous with respect to the ambient Haar measure.  We emphasize that we work in the dissipative setting where our generators $\{f_1, \dots, f_k\}$ are not assumed to preserve a common volume measure (although they are perturbations of volume-preserving diffeomorphisms).  Our hypotheses  also imply that each generator $f_i$ is Anosov and that the generators  satisfy a common cone condition.  

From a classification of all stationary measures we adapt the arguements of \cite{chung} to similarly classify all orbit closures (by showing all orbits are finite or dense). 

The arguments in this paper closely follow the arguments in the work of Tsujii, \cite{Tsujii-bigpaper}, where the author studied the existence and the absolute continuity of physical measures for partially hyperbolic endomorphisms on $\Tbb^{2}$ (see, also \cite{MR1862809}). 
\subsection{Setting and statement of the main theorem}\label{section.simplified}

Let $m$ be a smooth probability measure on $\mathbb{T}^2$, and let $\mathrm{Diff}^2_m(\mathbb{T}^2)$ be the set of $C^2$-diffeomorphisms preserving $m$.  Fix two diffeomorphisms $f,g\in \mathrm{Diff}^2_m(\mathbb{T}^2)$. Consider the following conditions: 
\begin{enumerate}
\item[\textbf{(C1)}] $f$ and $g$ are Anosov diffeomorphisms having a splitting $T\mathbb{T}^2  = E^s_{\star} \oplus E^u_\star$, for $\star= f,g$. 
\item[\textbf{(C2)}] There exist continuous cone fields $x \mapsto \mathcal{C}^s_x$ and $x\mapsto \mathcal{C}^u_x$,  Riemannian metrics $q^s$ and $q^u$ on $\TT^2$ and positive constants $0<\lambda_{s,-} < \lambda_{s,+}<1< 
\lambda_{u,-}< \lambda_{u,+}$ with the following property: for any $x\in \mathbb{T}^2$ and any non-zero vectors $v^s\in \mathcal{C}^s_{x}$ and $v^u  \in \mathcal{C}^u_x$,
\begin{itemize}
\item $Df^{-1}(x) \mathcal{C}^s_x \subset \mathcal{C}^s_{f^{-1}(x)}$ and $\lambda_{s,+}^{-1}\|v^s\|_{q^s} < \|Df^{-1}(x) v^s\|_{q^s}< \lambda_{s,-}^{-1}\|v^s\|_{q^s}$; and 
\item $Dg^{-1}(x) \mathcal{C}^s_x \subset \mathcal{C}^s_{g^{-1}(x)}$ and $\lambda_{s,+}^{-1}\|v^s\|_{q^s} < \|Dg^{-1}(x) v^s\|_{q^s} <\lambda_{s,-}^{-1}\|v^s\|_{q^s}$, 
where $\|\cdot\|_{q^s}$ denotes the norm induced by the Riemannian metric $q^s$.
\item $Df(x) \mathcal{C}^u_x \subset \mathcal{C}^u_{f(x)}$ and $\lambda_{u,-}\|v^u\|_{q^u} < \|Df(x) v^u\|_{q^u}< \lambda_{u,+}\|v^u\|_{q^u}$; and 
\item $Dg(x) \mathcal{C}^u_x \subset \mathcal{C}^u_{g(x)}$ and $\lambda_{u,-}\|v^u\|_{q^u}< \|Dg(x) v^u\|_{q^u}< \lambda_{u,+}\|v^u\|_{q^u}$, where $\|\cdot\|_{q^u}$ denotes the norm induced by the Riemannian metric $q^u$.
\end{itemize}
\item[\textbf{(C3)}] For every $x\in \mathbb{T}^2$,  $ E^u_f(x) \cap E^u_g(x) = \{0\}$.
\item[\textbf{(C4)}] For every $x\in \mathbb{T}^2$,  $ E^s_f(x) \cap E^s_g(x) = \{0\}$.
\end{enumerate}
Throughout this paper, we always assume that $f,g$ satisfies \textbf{(C1)} and \textbf{(C2)}. 
It is worth to mention that there are plenty of pairs $(f,g)$ of diffeomorphisms on $\Tbb^{2}$ that satisfies the conditions \textbf{(C1)} to \textbf{(C4)} as follows:

\begin{eg}
Let \[A=\begin{bmatrix} 2&1\\1&1 \end{bmatrix}, \textrm{ and } B=\begin{bmatrix} 3&5\\1&2 \end{bmatrix}.\] $A$ and $B$ induce toral automorphisms $L_{A}$ and $L_{B}$ on $\mathbb{T}^{2}$, respectively. We trivialize the tangent bundle $T\mathbb{T}^{2}$ to $\mathbb{T}^{2}\times \mathbb{R}^{2}$. It is easy to check that $L_{A}$ and $L_{B}$ satisfy conditions \textbf{(C1)} to \textbf{(C4)} with $\mathcal{C}^{s}=\{(x,y)\in \mathbb{R}^{2}:x<0<y \textrm{ or } y<0<x\}$ and $\mathcal{C}^{u}=\{(x,y)\in\mathbb{R}^{2}: 0<y<x \textrm{ or } x<y<0\}$. (Both $q^s$ and $q^u$ in \textbf{(C3)} and \textbf{(C4)} can be chosen as the standard product metric on $\TT^2$.)

 One can find more linear examples in toral automorphisms that satisfies conditions \textbf{(C1)} to \textbf{(C4)} as follows; Let $A,B$ be two hyperbolic matrices in $\textrm{GL}(2,\mathbb{Z})$. For hyperbolic matrix $L$ in $\textrm{GL}(2,\Rbb)$, let $E^{s}_{L}$ and $E^{u}_{L}$ be the eigenspace with eigenvalue smaller than $1$ and bigger than $1$, respectively. If two hyperbolic matrices $A$ and $B$ do not commute, either the pair $(A,B)$ or the pair $(A,B^{-1})$ has the property that there is an open cone $\mathcal{C}$ in $\Rbb^{2}$ such that $\mathcal{C}$ contains $E^{s}_{A}$ and $E^{s}_{B}$, and does not contain $E^{u}_{A}$ and $E^{u}_{B}$. This implies that for all sufficiently large $n$, the pair $(A^{n},B^{n})$ or $(A^{n},B^{-n})$ induces the pair of toral automorphisms satisfying conditions \textbf{(C1)} to \textbf{(C4)}. 

Also, it is easy to see that conditions \textbf{(C1)} to \textbf{(C4)} are $C^1$-open.  In particular, assume that $f,g\in \Diff^{2}_{m}(\mathbb{T}^{2})$ satisfies \textbf{(C1)} to \textbf{(C4)}. Then,  there are $C^1$-neighborhoods, $\widetilde{\mathcal{U}}_{f}$ and $\widetilde{\mathcal{U}}_{g}$,  of $f$ and $g$, respectively, in $\Diff^{1}(\mathbb{T}^{2})$ such that every pair $(\widetilde{f},\widetilde{g}) \in \widetilde{\mathcal{U}}_f \times \widetilde{\mathcal{U}}_g$ satisfies \textbf{(C1)} to \textbf{(C4)}. Hence, for instance, many non-linear examples can be found from the perturbation of linear examples.

 \end{eg}

\begin{dfn}\label{dfn.stationary}
Given a probability measure $\mu$ on $\mathrm{Diff}^2(\mathbb{T}^2)$, a probability measure $\nu$ on $\mathbb{T}^2$ is \emph{$\mu$-stationary}, if 
\[\nu=\mu*\nu:=\int_{\Omega} \left(f_{*}\nu\right) d\mu(f).\]
\end{dfn}
The operation $\mu*\nu$ is called the convolution of $\mu$ and $\nu$. Also, $\mu^{*n}*\nu$ is defined by $n$ times convolution. Our main theorem is about improving the SRB property to absolute continuity with respect to the Lebesgue class.

\begin{ltheorem}\label{thm.mainthm}
Let $f$ and $g$ verify the conditions \textbf{(C1)}-\textbf{(C3)} above. For any $\beta \in (0, \frac{1}{2}]$, there exist $C^2$-neighborhoods of $f$ and $g$ in $\mathrm{Diff}^2(\mathbb{T}^2)$, $\mathcal{U}_f$ and $\mathcal{U}_g$, with the following property: let $\mu$ be any probability measure on $\mathrm{Diff}^2(\mathbb{T}^2)$ such that $\mu(\mathcal{U}_f \cup \mathcal{U}_g) = 1$ and $\mu(\mathcal{U}_\star) \in [\beta, 1-\beta]$, for $\star=f,g$. Then, the unique $\mu$-stationary SRB measure $\nu$ is absolutely continuous with respect to $m$. Moreover, $\frac{d\nu}{dm}$ belongs to $L^2(m)$. 
\end{ltheorem}

The rest of our results uses the measure rigidity result by Brown and Rodriguez Hertz.  We will assume that $f$ and $g$ verify conditions \textbf{(C1)} - \textbf{(C4)} for Corollaries \ref{thm.thmforsurfaces} and \ref{thm.invariantmeasure}  below. 
Condition \textbf{(C3)} gives information about the oscillations of the unstable direction depending on the choice of past.  This condition allows us to improve the regularity of SRB measures, obtaining that they are absolutely continuous. Condition \textbf{(C4)} above gives information about the oscillation of the stable direction depending on the choice of future.   This is used to obtain measure rigidity results, thus classifying the possible stationary measures. Condition \textbf{(C4)} is related to a notion called \emph{uniform expansion} (see Section \ref{sec:equidistribution}) which has been used for obtaining several measure rigidity results in the random setting.

\begin{cor}\label{thm.thmforsurfaces}
Let $f$ and $g$ verify the conditions \textbf{(C1)}-\textbf{(C4)} above. Fix $\beta \in (0,\frac{1}{2}]$ and let $\mathcal{U}_f$ and $\mathcal{U}_g$ be given by Theorem \ref{thm.mainthm}. Let $\mu$ be a probability measure on $\mathrm{Diff}^2(\mathbb{T}^2)$  such that $\mu(\mathcal{U}_f \cup \mathcal{U}_g) = 1$, and $\mu( \mathcal{U}_\star) \in [\beta, 1-\beta]$, for $\star=f,g$. Then any ergodic $\mu$-stationary measure $\nu$ is either atomic or absolutely continuous with respect to $m$. 
\end{cor}

Another application  is the following. 

\begin{cor}\label{thm.invariantmeasure}
Let $f$ and $g$ verify conditions \textbf{(C1)} - \textbf{(C4)}. Fix $\beta \in (0,\frac{1}{2}]$ and let $\mathcal{U}_f$ and $\mathcal{U}_g$ be given by Theorem \ref{thm.mainthm}.  Suppose that $\nu$  is a non-atomic probability measure  such that $\nu$ is invariant by some diffeomorphism $\widehat{f} \in \mathcal{U}_f$ and by some diffeomorphism $\widehat{g} \in \mathcal{U}_g$. Then $\nu$ is absolutely continuous with respect to $m$.  
\end{cor}

 Given a set $S\subset \mathrm{Diff}^2(\mathbb{T}^2)$, let $\Gamma_S$ be the semigroup generated by $S$.  $\Gamma_S$ acts naturally on $\mathbb{T}^2$.  The $\Gamma_S$-orbit of a point $x\in \mathbb{T}^2$ is defined as the set $\{h(x): h\in \Gamma_S\}$. For \Cref{thm.equidistribution}, \Cref{thm.orbitclosureclassification}, and \Cref{thm.genericminimality} below, we will assume that $f$ and $g$ verify conditions \textbf{(C1)}, \textbf{(C2)} and \textbf{(C4)}. 

\begin{ltheorem}\label{thm.equidistribution}
Let $f$ and $g$ verify conditions \textbf{(C1)}, \textbf{(C2)} and \textbf{(C4)} above. There exist $C^2$-neighborhoods of $f$ and $g$, $\mathcal{U}_f$ and $\mathcal{U}_g$, with the following property.  Let $S$ be a finite subset of $\mathcal{U}_f  \cup \mathcal{U}_g$ and let $\mu$ be a probability measure such that $\mu(S) = 1$, $\mu(\mathcal{U}_f)$ and $\mu(\mathcal{U}_g) >0 $,  and let $\nu$ be the unique $\mu$-stationary SRB measure.  Suppose that  $x\in \mathbb{T}^2$ has infinite  $\Gamma_S$-orbit. Then, 
\[
\displaystyle \lim_{n\to +\infty} \frac{1}{n} \sum_{j=0}^{n-1} \left(\mu^{*j}*\delta_x \right) =\nu,
\] 
where the convergence is in the weak*-topology. 
\end{ltheorem}

\begin{cor}\label{thm.orbitclosureclassification}
Under the same assumptions as Theorem \ref{thm.equidistribution}, let $\mu$ be a probability measure such that $\mu(S)=1$, $\mu(\mathcal{U}_f)$ and $\mu(\mathcal{U}_g)>0$.  Then every $\Gamma_S$-orbit is either finite or dense. 
\end{cor}

As an application of \Cref{thm.orbitclosureclassification}, we obtain the following result.

\begin{cor}\label{thm.genericminimality}
For any $\widehat{g} \in \mathcal{U}_g$, there exists a dense $\text{G}_\delta$ subset of  $\mathcal{U}_f$,  $\mathcal{R}_{\widehat{g}}$, with the following property. For any  $\widehat{f} \in \mathcal{R}_{\widehat{g}}$, define $S =\{ \widehat{f}, \widehat{g}\}$ and let $\Gamma_S$ be the semigroup generated by $S$. Then, the $\Gamma_S$-action is minimal, that is, every $\Gamma_S$ orbit is dense. 
\end{cor}

\subsection*{Acknowledgments}

A.\ B.\ was partially supported by the  National Science Foundation under Grant Nos.\ DMS-2020013 and DMS-2400191.  H.\ L.\ was supported by an AMS-Simons Travel Grant. D.\ O.\ was partially supported by the National Science Foundation under Grant No. \ DMS-2349380. 

\section{Preliminaries}

\subsection{Skew extension and stationary measure}\label{subsec:skew} We recall facts on random dynamical systems on smooth manifolds.  We mainly deal with random dynamical systems in the setting of \Cref{thm.mainthm}. Most of the arguments can be found in many literatures, such as \cite{Liu-Qian-book}.

Let $M$ be a smooth manifold. Consider $\Diff^{2}(M)$ with the $C^2$-topology and denote by $\Bcal(\Diff^{2}(M) )$ the Borel $\sigma$-algebra on $\Diff^2(M)$. Note that $\Diff^{2}(M)$ is a Polish space. Let $\mu$ be a probability measure on $(\Diff^{2}(M) ,\Bcal(\Diff^{2}(M) ))$. When we have a probability measure on this space, we always consider the completion of the $\sigma$-algebra with respect to the measure and still denote the completion of $\sigma$-algebra by the same notation.

Let $\Omega^{+}=(\Diff^2(M))^{\Nbb}$ and $\Omega = (\Diff^2(M))^{\mathbb{Z}}$.  Consider $\Omega^{+}$ equipped with the Borel probability $\mu^{\Nbb}$ which is an infinite product of $\mu$ and the ($\mu^{\Nbb}$ completion of) Borel $\sigma$ algebra $\Bcal(\Diff^2(M))^{\Nbb}$. For each $\omega\in \Omega^{+}$, $\omega=\left(f_{0}, f_1, f_2, \cdots \right)$, we define 
\[f_{\omega}^{0}=id, f_{\omega}^{n}=f_{n-1} \circ\dots\circ f_{0}\quad\textrm{ for } n\ge 1.\]
Moreover, if $\omega = \left((\cdots, f_{-2}, f_{-1}, f_0, f_1, f_2 \cdots \right)\in \Omega$, then
\[
f^{-n}_\omega = (f_{-n})^{-1} \circ \cdots (f_{-1})^{-1}\quad\textrm{ for } n\geq 1.
\]
We remark that $(f_\omega^n)^{-1}$ is defined for one sided words, and it is different from $f_\omega^{-n}$.

Naturally, we can consider a skew product related to the random dynamical system $F^{+}\colon\Omega^{+}\times M \to \Omega^{+}\times M$ as \[F^{+} \left(\omega, x\right)= \left(\sigma(\omega),f_{\omega}(x)\right),\] where $\sigma:\Omega^{+}\to \Omega^{+}$ is the (left) shift map and $f_{\omega}=f^{1}_{\omega}$. 

For Claim \ref{claim.1} and Proposition \ref{prop.basic} below, see Chapter $1$ in \cite{Liu-Qian-book}. 

\begin{claim}\label{claim.1}
$\nu$ is a $\mu$-stationary measure if and only if $\mu^{\Nbb}\otimes \nu$ is $F^{+}$-invariant. Furthermore, $\nu$ is $\mu$ ergodic stationary measure if and only if $\mu^{\Nbb}\otimes \nu$ is $F^{+}$-ergodic invariant measure. 
\end{claim}

We can consider the natural extension of $F^+$, which is the map $F:\Omega \times M \to \Omega \times M$ defined in the same way as $F^+$.

\begin{prop}\label{prop.basic}
Given a $\mu$-stationary measure $\nu$, there exists a unique Borel probability measure $\widehat{\nu}$ on $\Omega\times M$ such that 
\begin{enumerate}
\item $\widehat{\nu}$ is $F$-invariant, and
\item $P_{*}^{+}\left(\widehat{\nu}\right)=\mu^{\Nbb}\otimes \nu$, where $P^{+}\colon \Omega\times M\to \Omega^{+}\times M$ is the natural projection.
\end{enumerate}
Furthermore, if we disintegrate $\widehat{\nu}$ with respect to $P\colon \Omega\times M\to \Omega$, there is a family of Borel probability measure $\left\{\nu_{\omega}\right\}_{\omega\in \Omega}$ such that \[\widehat{\nu}=\int_{\Omega} \nu_{\omega} d\mu^{\mathbb{Z}}(\omega).\] Moreover, for $\mu^{\mathbb{Z}}$-almost every $\omega=(\dots, f_{-1}, f_{0},f_{1},\dots)$, $\nu_{\omega}$ only depends on $\omega^{-}=\left(\dots, f_{-2},f_{-1}\right)$.
\end{prop}
We call the probability measure $\nu_{\omega}=\nu_{\omega^{-}}$ on $M$ a \emph{sample measure} with respect to $\omega$.

\subsection{Stable and Unstable manifold}\label{sec:stablemfd}  In this subsection, we assume that $f$ and $g$ satisfy conditions \textbf{(C1)} and \textbf{(C2)} above. 

Then, we can choose sufficiently small $C^{1}$-neighborhoods $\Ucal_{f}$ and $\Ucal_{g}$ of $f$ and $g$ in $\Diff^{2}(\Tbb^{2})$ so that,  for any $x\in\TT^2$, $n\in\ZZ_{\geq 0}$,  $\omega\in\cU^{\ZZ}$,  vectors $v^s\in\cC^s_x$ and $v^u\in\cC_x^u$, 
\begin{itemize}
\item $Df^{-n}_\omega(x) \mathcal{C}^s_x \subset \mathcal{C}^s_{f^{-n}_\omega(x)}$ and {$(\C{C''}{0})^{-1}\lambda_{s,+}^{-n}\|v^s\| < \|Df^{-n}_\omega(x) v^s\|< \C{C''}{0}\lambda_{s,-}^{-n}\|v^s\|$}; 
\item $Df^n_\omega(x) \mathcal{C}^u_x \subset \mathcal{C}^u_{f^n_\omega(x)}$ and {$(\C{C''}{0})^{-1}\lambda_{u,-}^{n}\|v^u\| < \|Df^{n}_\omega(x) v^s\|< \C{C''}{0}\lambda_{u,+}^{n}\|v^u\|$},
\end{itemize}
where $\Ucal=\Ucal_{f}\cup \Ucal_{g}$.

\subsubsection{Uniform hyperbolicity}\label{sec:mu-axioms}

 Let $\mu$ be a probability measure supported on $\Ucal$.  The following moment condition holds automatically: \begin{equation*}
\int_{\Omega}\left( \log^{+}||f||_{C^{2}} +\log^{+}||f^{-1}||_{C^{2}} \right) d\mu(f) <\infty 
\end{equation*}
where $\log^{+}(x)=\max\{x,0\}$ and $||\cdot||_{C^{2}}$ is the $C^{2}$-norm of a diffeomorphism.

Consider the skew products defined in Section \ref{subsec:skew}. We will restrict these skew products to $\mathcal{U}^{\mathbb{Z}} \times \mathbb{T}^2$ and $\mathcal{U}^{\mathbb{N}} \times \mathbb{T}^2$.  Because of the joint cone condition, we can observe that  the skew product is uniformly hyperbolic on the fibers. Indeed, the joint cone condition let us  define stable and unstable distribution for every point $x\in \Tbb^{2}$ uniformly as follows.

\begin{prop}\label{prop:UH}
Under the setting above, for every word $\omega\in \mathcal{U}^{\mathbb{Z}}$, there exists a (continuous) splitting $T\Tbb^{2}=E^{s}_{\omega}\oplus E^{u}_{\omega}$ and constants $0<\gamma<1$, $C>0$, $\CS{L}{0}>1$, $0<\theta<1$ such that 
\begin{enumerate}
\item $Df_{\omega_{0}}E^{s}_{\omega}=E^{s}_{\sigma(\omega)}$ and $Df_{\omega_{0}}E^{u}_{\omega}=E^{u}_{\sigma(\omega)}$
\item For $v^{s}\in E^{s}_{\omega}$ and $v^{u}\in E^{u}_{\omega}$, we have
\[ ||Df_{\omega}^{n}v^{s}|| <C\gamma^{n}||v^{s}|| \textrm{ and } ||Df_{\omega}^{-n}v^{u}||<C\gamma^{n}||v^{u}||\] for all $n\ge 0$

\item for all $\omega\in \mathcal{U}^{\mathbb{Z}}$, $x\mapsto E^{s}_{\omega,x}$ and $x\mapsto E^{u}_{\omega,x}$ are $(\C{L}{0},\theta)$-H\"older continuous,
\item for all $\omega\in \mathcal{U}^{\mathbb{Z}}$ and for all $x\in \Tbb^{2}$, $\sphericalangle(E^{u}_{\omega,x},E^{s}_{\omega,x})>\alpha$.
\end{enumerate}
\end{prop}
The proof is basically the same as in the single Anosov diffeomorphism case (see, for instance, \cite[Chapter 6]{KH}). From \Cref{prop:UH}, we can get the local stable and unstable manifolds for every word $\omega\in \mathcal{U}^{\mathbb{Z}}$ and points $x\in \Tbb^{2}$ using the graph transform method (see, for instance, \cite[Theorem 6.2.8]{KH}).
Let \[W_{r}^{s}(\omega,x)=\left\{y\in \Tbb^{2}: d(f^{n}_{\omega}(y),f^{n}_{\omega}(x))\le r, \textrm{ for all $n\ge 0$ and }\lim_{n\to\infty}d(f^{n}_{\omega}(y),f^{n}_{\omega}(x))=0\right\}\] and
\[W_{r}^{u}(\omega,x)=\left\{y\in \Tbb^{2}: d(f^{-n}_{\omega}(y),f^{-n}_{\omega}(x))\le r,  \textrm{ for all $n\ge 0$ and } \lim_{n\to\infty}d(f^{-n}_{\omega}(y),f^{-n}_{\omega}(x))=0\right\}\]
\begin{prop}\label{prop:stablemfd}
Under the above setting, there is $r>0$ such that for all $\omega\in \mathcal{U}^{\mathbb{Z}}$, 
\begin{enumerate}
\item $W_{r}^{s}(\omega,x)$ and $W_{r}^{u}(\omega,x)$ is a $C^{2}$ embedded curve tangent to $E^{\star}_{\omega,x}$. 
\item $W_{r}^{\star}(\omega,x)$ is continuous in $x$ with respect to the $C^{2}$ topology, for $\star =s,u$.
\item There exists $C\ge 1$ and $0<\lambda<1$ such that $W_{r}^{\star}(\omega,x)$ can be characterized by 
\[W_{r}^{s}(\omega,x)=\left\{y\in \Tbb^{2}:	d(f^{n}_{\omega}(y),f^{n}_{\omega}(x))\le r\textrm{  and } d(f^{n}_{\omega}(y),f^{n}_{\omega}(x))\le C\lambda^{n}d(f^{n}_{\omega}(y),f^{n}_{\omega}(x))	\textrm{ for all $n\ge 0$}		\right\}\]
\[W_{r}^{u}(\omega,x)=\left\{y\in \Tbb^{2}:d(f^{-n}_{\omega}(y),f^{-n}_{\omega}(x))\le r\textrm{  and } d(f^{-n}_{\omega}(y),f^{-n}_{\omega}(x))\le C\lambda^{n}d(f^{-n}_{\omega}(y),f^{-n}_{\omega}(x))	\textrm{ for all $n\ge 0$}.		\right\}\]
\end{enumerate}
\end{prop}

Indeed, for each $\omega\in \mathcal{U}^{\mathbb{Z}}$ and $x\in \Tbb^{2}$, there exists a $C^{2}$ function $\phi_{\omega,x}^{s}:E^{s}_{\omega,x}(r)\to E^{u}_{\omega,x}$ so that $\phi_{\omega,x}^{s}(0)=0$, $D_{0}\phi_{\omega,x}^{s}(0)=0$ and $W_{r}^{s}(\omega,x)=\exp(\textrm{graph}\phi_{\omega,x}^{s})$ where $E^{s}_{\omega,x}(r)=\{v\in E^{s}_{\omega,x}:||v||<r\}$. The same  holds for unstable manifolds.

The global stable and unstable manifolds are defined by 
\begin{align*}
W^{s}(\omega,x)&=\bigcup_{n\ge 0}f_{\omega}^{-n}W^{s}(\sigma^{n}\omega, f^{n}_{\omega}(x))=\left\{y\in \Tbb^{2}:	\lim_{n\to\infty} d(f^{n}_{\omega}(y),f^{n}_{\omega}(x))=0	\right\}\end{align*}

\begin{align*}
W^{u}(\omega,x)&=\bigcup_{n\ge 0}f_{\omega}^{n}W^{u}(\sigma^{-n}\omega, f^{-n}_{\omega}(x))=\left\{y\in \Tbb^{2}:	\lim_{n\to-\infty} d(f^{n}_{\omega}(y),f^{n}_{\omega}(x))=0	\right\}\end{align*}

\subsubsection{Random SRB measure}\label{sec:SRBdef} 
 Recall that given a $\mu$-stationary measure $\nu$, we can construct an $F$-invariant probability measure $\widehat{\nu}$ on $\mathcal{U}^{\mathbb{Z}}\times \Tbb^{2}$ as in \cref{prop.basic}. Let us fix a $\mu$-stationary measure, $\nu$ and let $\hat{\nu}$ be its lift.

\begin{dfn}
A $\hat{\nu}$-measurable partition $\eta$ of $\mathcal{U}^{\mathbb{Z}}\times \Tbb^{2}$ is said to be subordinated to $W^{u}$ manifolds if 
for $\widehat{\nu}$-almost every $(\omega,x)$, $\{y\in \Tbb^{2}:(\omega,y)\in \eta(\omega,x)\}$ is 
\begin{enumerate}
\item precompact in $W^{u}(\omega,x)$, 
\item contained in $W^{u}(\omega,x)$, and
\item  contains an open neighborhood of $x$ in $W^{u}(\omega,x)$.
\end{enumerate}
\end{dfn}
Note that such a measurable partition always exists.
Let $\hat{\nu}_{(\omega,x)}^{\eta}$ be a system of conditional measures with respect to a $W^{u}$- subordinated $\hat{\nu}$-measurable partition $\eta$. 

\begin{dfn}[Random SRB]  A $\mu$-stationary measure $\nu$ has the SRB property if for every $W^{u}$-subordinated $\hat{\nu}$-measurable partition $\eta$,  for $\mu^{\mathbb{Z}}$-almost every $\omega$ and $\nu_\omega$-almost every $x$,  the measure $\hat{\nu}_{(\omega,x)}^{\eta}$ is absolutely continuous with respect to the Lebesgue measure on $W^{u}(\omega,x)$ inherited by the immersed Riemannian submanifold structure on $W^{u}(\omega,x)$.
\end{dfn}

\begin{lem}\label{lem.fullsupport}
Let $\mu$ be a probability measure supported on $\mathcal{U}$ and suppose that $\nu$ is a $\mu$-stationary SRB measure. Then $\text{supp}(\mu) = \mathbb{T}^2$.
\end{lem}
\begin{proof}
Let us first show the following claim.
\begin{claim}\label{claim.stationaryinclusion}
Suppose that $\nu'$ is a $\mu$-stationary measure. Then 
\[\displaystyle \bigcup_{h\in \text{supp}(\mu)} h(\text{supp}(\nu')) \subset \text{supp}(\nu').\]
\end{claim}
\begin{proof}[Proof of Claim \ref{claim.stationaryinclusion}]
Take $x\in \displaystyle \bigcup_{h\in \text{supp}(\mu)} h(\text{supp}(\nu')$, then, there exists $\widehat{h}\in \text{supp}(\mu)$ such that $\widehat{h}^{-1}(x) \in \text{supp}(\nu')$.  In particular, for any $r,\delta>0$, we have
\[
\displaystyle \int_{B(\widehat{h},\delta)} \nu'(B(h^{-1}(x), r)) d\mu(h) >0.
\]
For each $R>0$, there exists $r>0$ such that $h^{-1}(B(x,R)) \supset B(h^{-1}(x),r)$, for every $h\in \text{supp}(\mu)$.   Therefore,
\[
\displaystyle \nu'(B(x,R)) = \int_{\mathcal{U}'} \nu'(h^{-1}(B(x,R))) d\mu(h) \geq  \int_{B(\widehat{h},\delta)} \nu'(B(h^{-1}(x), r)) d\mu(h) >0.
\]
Since this is true for any $R>0$, we have that $x\in \text{supp}(\nu')$. \qedhere
\end{proof}

Suppose that $\nu$ is a $\mu$-stationary SRB measure. In particular, the support of $\nu$ contains a curve $\gamma$ tangent to the unstable cone.  Take $h\in \text{supp}(\mu)$. By Claim \ref{claim.stationaryinclusion} and by induction, we obtain that $h^n(\gamma) \subset \text{supp}(\nu)$ for every $n\in \mathbb{N}$.  Observe that $h$ is an Anosov diffeomorphism, in particular,  the unstable foliation is minimal.  For any $\varepsilon>0$, there exists $L>0$ such that any unstable leaf for $h$ of length $L$ is $\varepsilon$-dense.  For each $n$ large enough, there exists $D_n \subset \gamma$ such that $h^n (D_n)$ is $\varepsilon$-close to an unstable manifold of length $L$. In particular $h^n(D_n)$ is $2\varepsilon$-dense. It is easy to conclude that this implies that $\text{supp}(\nu) = \mathbb{T}^2$. \qedhere

\end{proof}

We can ensure that, in our setting, there is a unique $\mu$-stationary SRB measure $\nu$ as follows:

\begin{theorem}[\cite{Liu-Qian-book}] \label{thm.uniquesrb}
Let $\mu$ be a probability measure $\mu$ supported on $\mathcal{U}$. Then there exists a unique $\mu$-stationary SRB measure $\nu$. 
\end{theorem}
\begin{proof}
The proof follows the same steps of the proof of Theorem $1.1$ in Chapter VII of \cite{Liu-Qian-book}. Even though in their setting the authors work with random perturbations of a single system, the key feature to make the proof work is uniform hyperbolicity for any point and any choice of word $\omega$, which we have in our setting.  The proof the follows the following steps. Consider any disk $D^u$ tangent to $\mathcal{C}^u$. The riemannian metric of $\mathbb{T}^2$ induces a riemannian volume on $D^u$. Let $m^u$ be the normalized volume measure on $D^u$.  For each $n\in \mathbb{N}$, consider 
\[
\nu_n := \displaystyle \frac{1}{n} \sum_{j=0}^{n-1} \mu^{*j}* m^u.
\]
Since the skew product is uniformly hyperbolic, one obtains bounded distortion estimates. This implies that any accumulation measure of the sequence $(\nu_n)_{n\in \mathbb{N}}$ is a $\mu$-stationary measure having  the SRB property.   This implies the existence part of the statement.

Suppose there are two different ergodic $\mu$-stationary SRB measures $\nu$ and $\nu'$.  By Lemma \ref{lem.fullsupport}, and by using that for any choice of past, the stable and unstable manifolds have uniform size, one can find homoclinic relations between the two measures and then apply a Hopf argument to conclude that $\nu = \nu'$.  See Lemma $3.1$ and Proposition $3.4$ in Chapter VII of \cite{Liu-Qian-book} for more details on the Hopf argument. 
\end{proof}

The goal of Theorem \ref{thm.mainthm} is to show that the $\mu$-stationary SRB measure $\nu$ is actually absolutely continuous with respect to the Lebesgue measure.  For a $W^{u}$-subordinated $\hat{\nu}$-measurable partition $\eta$, we denote by $\eta_{\omega}(x)$ the set $\{y\in \Tbb^{2}:(\omega,y)\in \eta(\omega,x)\}$. Note that for each $\omega$, $\eta_{\omega}(x)$ forms a $\nu_\omega$-measurable partition on $\Tbb^{2}$. 

\begin{theorem}[Log-Lipschitz regularity of the density]\label{thm.log.lip.density}
Let $\mu$ be as a probability measure as in \cref{sec:mu-axioms} and let $\nu$ be the unique $\mu$-stationary measure on $\Tbb^{2}$ with the SRB property. 
Let $m^{u}_{(\omega,x)}$ denote the Lebesgue measure on $W^{u}(\omega,x)$ induced by the immersed Riemannian structure on $W^{u}(\omega,x)$. Fix a  $W^{u}$-subordinated $\hat{\nu}$-measurable partition $\eta$ of $\mathcal{U}^{\mathbb{Z}}\times \Tbb^2$. 

Then, for $\widehat \nu$-a.e.\ $(\omega,x)$, there exists 
a log-Lipschitz function $h^\eta_{\omega,x}\colon  \eta_\omega(x) \to \Rbb^{+}$ 
such that 
\[h^\eta_{\omega,x}(y)=\frac{d\hat{\nu}_{(\omega, x)}^{\eta}}{dm^{u}_{(\omega,x)}}(y)\] for $m_{(\omega,x)}^{u}$-almost every $y\in \eta_{\omega}(x)$.  Moreover, the log-Lipschitz constant is uniform over the choice of $W^u$-subordinated partition $\eta$ and $(\omega,x)\in \mathcal{U}^{\mathbb{Z}}\times \Tbb ^2$.

\end{theorem}
Indeed, for $y\in \eta_\omega (x)$, 
we let $$J_{\omega,x}(y)= \lim_{n\to \infty} \frac{ \|{D_yf_\omega^{-n}\restrict {E^u_{\omega,y}}}\|}{ \|{D_xf_\omega^{-n}\restrict {E^u_{\omega,x}}}\|}$$
and 
$$h^\eta_{\omega,x}(y)= \frac{1}{\int_{\eta_\omega(x)} J_{\omega,x}(y) \ d m^u_{(\omega,x)}(y)} J_{\omega,x}(y).$$
A standard computation (see \cite[Corollary 6.1.4]{LY1}) shows that $h^\eta_{\omega,x}(y)= \frac{d\hat{\nu}_{(\omega,x)}^{\eta}}{dm^{u}_{(\omega,x)}}(y).$
Moreover,
since the Lipschitz variation of $(\omega,y)\mapsto \|{D_yf_\omega^{-1}\restrict {E^u_{\omega,y}}}\|$ along $W^u(\omega, x)$ is uniform, (independent of $(\omega,x)$), 
and since for $y,z\in \eta_\omega(x)$, $d(f_\omega ^{-n}(y), f_\omega ^{-n}(z))\to 0$ exponentially fast (uniformly in $\omega, y$, and $z$), there is $L$ (independent of $\eta$ and $(\omega, x)$) such that 
$$|\log h^\eta_{\omega,x}(y)- \log h^\eta_{\omega,x}(z)| = \log \frac {J_{\omega,x}(y)}{J_{\omega,x}(z)} \le L d(y,z).$$

\subsection{Other notations}\label{subsect:other.no.}
We introduce some notations and conventions which will be used throughout the paper. When we introduce new constants in the rest of this paper, we do not always track their dependence on the constants introduced in this section ($\C{C}{0}$, $\C{C'}{0}$ and $\C{C''}{0}$) and on certain constants introduced earlier ($\lambda_{\star,\pm}$ in \textbf{(C2)}; $\C{L}{0}$ and $\theta$ in Item (3) of Proposition \ref{prop:UH}).

\textbf{Notations regarding $\TT^2$}:
\begin{enumerate}
\item We identify $\TT^2$ with $\RR^2/\ZZ^2$.
\item Let $T\mathbb{T}^{2}$ and $\mathbb{P}T\mathbb{T}^{2}$ be the tangent bundle and the projective tangent bundle of $\mathbb{T}^{2}$, respectively. 
\item We fix a smooth trivialization $T\mathbb{T}^{2}\simeq \mathbb{T}^{2}\times \mathbb{R}^{2}$ and $\mathbb{P}T\mathbb{T}^{2}\simeq \mathbb{T}^{2}\times \mathbb{R}P^{1}$. 
\item We fix a standard inner product structure on $\mathbb{R}^{2}$ with an orthonormal basis $\{e_{1}:=(1,0),e_{2}:=(0,1)\}$. This induces a smooth Riemannian metric on $\TT^2$. We refer to this Riemannian metric as the \emph{standard Riemannian metric} on $\TT^2$. For simplicity, we denote by $d(\cdot, \cdot)$ the induced distance function on $T\mathbb{T}^{2}$ and the induced distance function on $\mathbb{T}^{2}$.  (In particular, the injectivity radius of $\TT^2$ equipped with the standard Riemannian metric is $1/2$.) 

Even though, we use the same notation $d$ for metrics on different spaces, it will be clear in the context.
\item Unless otherwise stated, unit vectors in $T\TT^2$ always means unit vectors in $T\TT^2$ with respect to the standard Riemannian metric. Similarly, orthogonality in $T\TT^2$ always means orthogonality in $T\TT^2$ with respect to the standard Riemannian metric. For any $v\in T\TT^2$, $\|v\|$ always denotes the norm of $v$ with respect to the standard Riemannian metric. The curvature of a $C^2$-curve $\gamma(t)$ on $\TT^2$ always refers to the curvature with respect to the standard Riemannian metric, which is given by $\det(\dot\gamma(t),\ddot\gamma(t))/\|\dot\gamma(t)\|^3$.
\item We fix a standard distance on $\mathbb{R}P^{1}$ given by the angle and the induced metric, denoted also by $d(\cdot,\cdot)$, on $\mathbb{P}T\mathbb{T}^{2}$. 
\end{enumerate}

\textbf{Notations regarding the smooth measure $m$}:
\begin{enumerate}
\item[(7)] We denote by $\Leb$ the probability measure on $\TT^2$ induced by the standard Riemannian metric on $\TT^2$. Since $m$ is a smooth probability measure on $\TT^2$, we fix $\CS{C}{0}\geq 1$ such that $\C{C}{0}^{-1}\leq dm/d\Leb\leq \C{C}{0}$.
\end{enumerate}

\textbf{Notations regarding $C^2$-norm of $f$ and $g$}:
\begin{enumerate}
\item[(8)] We fix a constant $\CS{C'}{0}>0$ such that $\max\{\|f\|_{C^2},\|f^{-1}\|_{C^2},\|g\|_{C^2},\|g^{-1}\|_{C^2}\}\leq \C{C'}{0}$.
\end{enumerate}

\textbf{Additional notations regarding conditions (C1) and (C2)}:
\begin{enumerate}
\item[(9)] Throughout this paper, whenever we choose neighborhoods $\cU_f$ of $f$ and $\cU_g$ of $g$ in $\mathrm{Diff}^2(\TT^2)$, we assume that for any $\widetilde{f}\in\cU_f$ and $\widetilde{g}\in\cU_g$, the following holds:
\begin{itemize}
\item The pair $(\widetilde{f},\widetilde{g})$ satisfies \textbf{(C1)} and \textbf{(C2)} (with respect to the same choice of cone fields, $q^\star$ and $\lambda_{\star,\pm}$ for the pair $(f,g)$, where $\star=s,u$). 
\item $\max\{\|\widetilde{f}\|_{C^2},\|\widetilde{f}^{-1}\|_{C^2},\|\widetilde{g}\|_{C^2},\|\widetilde{g}^{-1}\|_{C^2}\}\leq \C{C'}{0}$.
\end{itemize}
One can check easily that for any fixed choice of cone fields, $q^\star$ and $\lambda_{\star,\pm}$ with $\star=s,u$, conditions \textbf{(C1)-(C4)} are open. Hence any sufficiently small neighborhoods $\cU_f$ and $\cU_g$ satisfies the above two bullet points.
\item[(10)] Let $\CS{C''}{0}=\sqrt{\sup_{\star=s,u}\{\|\cdot\|_{q^\star}/\|\cdot\|, \|\cdot\|/\|\cdot\|_{q^\star}\}}\geq 1$. Then for any neighborhoods $\cU_f$ of $f$ and $\cU_g$ of $g$ in $\mathrm{Diff}^2(\TT^2)$ satisfying (9), the condition \textbf{(C2)} implies the following: let $\cU=\cU_f\cup\cU_g$. Then for any $x\in\TT^2$, for any $n\in\ZZ_+$, for any $\omega\in\cU^{\ZZ}$ and for any vector $v^s\in\cC^s_x$ and $v^u\in\cC_x^u$,
\begin{itemize}
\item $Df^{-n}_\omega(x) \mathcal{C}^s_x \subset \mathcal{C}^s_{f^{-n}_\omega(x)}$ and {$(\C{C''}{0})^{-1}\lambda_{s,+}^{-n}\|v^s\| < \|Df^{-n}_\omega(x) v^s\|< \C{C''}{0}\lambda_{s,-}^{-n}\|v^s\|$}; 
\item $Df^n_\omega(x) \mathcal{C}^u_x \subset \mathcal{C}^u_{f^n_\omega(x)}$ and {$(\C{C''}{0})^{-1}\lambda_{u,-}^{n}\|v^u\| < \|Df^{n}_\omega(x) v^s\|< \C{C''}{0}\lambda_{u,+}^{n}\|v^u\|$}.
\end{itemize}
\end{enumerate}

\section{The semi-norm $\|.\|_{\rho}$}

Given two finite measures $\nu$ and $\nu'$ on $\mathbb{T}^2$, and a number $\rho>0$, we define the $\rho$-inner product between $\nu$ and $\nu'$ by
\begin{align*}
\langle \nu, \nu'\rangle_{\rho} := \frac{1}{\rho^{4}}\int_{\mathbb{T}^2} \nu(B(z,\rho))\nu'(B(z,\rho))  dm(z),
\end{align*}
where $B(z,\rho)$ denotes the ball (with respect to the standard product metric on $\TT^2$) centered at $z$ with radius $\rho$. Define the $\rho$-semi-norm of $\nu$ by $\|\nu\|_{\rho} = \sqrt{\langle \nu, \nu\rangle_{\rho}}$. 

\begin{lem}[\cite{Tsujii-bigpaper}, Lemma 6.2]\label{lem.liminfnorm}
If $\liminf_{\rho \to 0} \|\nu\|_{\rho} < +\infty$ then $\nu$ is absolutely continuous with respect to the smooth measure $m$ and $\lim_{\rho \to 0} \|\nu\|_{\rho} = \left\| \frac{d\nu}{dm}\right\|_{L^2(m)}$.
\end{lem}

We will also need the following lemma.

\begin{lem}[\cite{Tsujii-bigpaper}, Lemma 6.1]
\label{lem.changeofscale}
There is a constant $\CS{C}{1}>1$, such that  for any $0< \rho \leq \delta <1$
\[
\|\nu\|_{\delta} \leq \C{C}{1}\|\nu\|_{\rho}.
\]
\end{lem}

\begin{lem}[\cite{Tsujii-bigpaper}, Lemma 6.3]
\label{lem.normconvergence}
If a sequence of Borel finite measures $\nu_k$ converges weakly to a measure $\nu_{\infty}$, then for any $\rho>0$, we have $\|\nu_\infty\|_\rho = \lim_{k\to +\infty} \|\nu_k\|_\rho$.
\end{lem}
The above semi-norm can be generalized by allowing $\rho$ to be a positive function on $\TT^2$. To be specific, for any $\nu\in\mathrm{Prob}(\TT^2)$ and any Lebesgue measurable, strictly positive function $r:\TT^2\to\RR_+$, we define
\begin{align}\label{generalized r-norm}
\|\nu\|_r^2:=\int_{\TT^2}\frac{(\nu(B(x,r(x))))^2}{(r(x))^4}dm(x).
\end{align}
In particular, for any $\rho\in\RR_+$, $\|\nu\|_\rho$ is given by \eqref{generalized r-norm} with $r(x)\equiv \rho$. We present a generalization of Lemma \ref{lem.changeofscale} by allowing $\delta$ to vary over points on $M$.

\begin{lem}\label{generalized large<small}
There exist a constant $\CS{C}{2}>1$ such that the following holds. For any positive numbers $0<\rho\leq \delta_-\leq \delta_+\leq 1$ and any Lebesgue measurable function $\delta:\TT^2\to \RR$ such that $\delta(\TT^2)\subset [\delta_-,\delta_+]$, we have 
$$\|\nu\|^2_\delta\leq \C{C}{2}(1+\log(\delta_+/\delta_-))\|\nu\|^2_\rho.$$
\end{lem}
\begin{proof}
Let $A_\rho:=\{z_1,...,z_k\}$ be a maximal $(\rho/5)$-separated subset, that is, a maximal subset of $\TT^2$ (with respect to inclusion) such that for any $x\neq y\in A_\rho$, $d(x,y)>(\rho/5)$. Then there exist some $N_0>0$ independent of the choice of $\rho$, such that 
\begin{enumerate}
\item $\bigcup_{z\in A_\rho}B(z,\rho/4)=\TT^2$;
\item For any $x\in \TT^2$, there are at most $N_0$ points $z\in A_\rho$ such that $x\in B(z,\rho/4)$.
\end{enumerate}
Indeed, if there exists some $x\in \left(\bigcup_{z\in A_\rho}B(z,\rho/4)\right)\setminus M$, then $A_\rho\sqcup\{x\}$ is a strictly larger $(\rho/5)$-separated subset of $M$. This contradicts the maximality of $A_\rho$. For any $x\in M$, $A_\rho\cap B(x,\rho/4)$ is $(\rho/5)$-separated. Therefore $\{B(z,\rho/13)\}_{z\in A_\rho\cap B(x,\rho/4)}$ is a collection of disjoint subsets of $B(x, \rho/3)$. Hence $|A_\rho\cap B(x,\rho/4)|\cdot(\pi\rho^2/169)\leq \pi\rho^2/9$. In particular, there are at most $N_0:=19=[169/9]+1$ points $z\in A_\rho$ such that $x\in B(z,\rho/4)$.

By the second property of $A_\rho$ from above and (7) in Section \ref{subsect:other.no.}, we have
\begin{align*}
\|\nu\|^2_\rho=&\frac{1}{N_0^2\rho^4}\int_{\TT^2}\left(N_0\nu(B(x,\rho))\right)^2dm(x)\\
\geq &\frac{1}{N_0^2\rho^4}\int_{\TT^2}\left(\sum_{z\in A_\rho\cap B(x,\rho/2)}\nu(B(z,\rho/4))\right)^2dm(x)\\
\geq &\frac{1}{N_0^2\rho^4}\int_{\TT^2}\sum_{z\in A_\rho\cap B(x,\rho/2)}\left(\nu(B(z,\rho/4))\right)^2dm(x)\\
=&\frac{1}{N_0^2\rho^4}\sum_{z\in A_\rho}\left(\nu(B(z,\rho/4))\right)^2\left(\int_{B(z,\rho/2)}dm(x)\right)\geq\frac{\pi}{4N_0^2\rho^2\C{C}{0}}\sum_{z\in A_\rho}\left(\nu(B(z,\rho/4)))\right)^2.\\
\end{align*}

Therefore it suffices to show that there exists $\C{C}{2}>1$ independent of the choice of $\delta$ and $\rho$, such that 
\begin{align}\label{rho-discretize}
\|\nu\|^2_\delta\leq \frac{\C{C}{2}\pi(1+\log(\delta_+/\delta_-))}{4N_0^2\rho^2\C{C}{0}}\sum_{z\in A_\rho}\left(\nu(B(z,\rho/4)))\right)^2.
\end{align}

By the first property of $A_\rho$, we have 
\begin{align}\label{delta<rho-1}
\|\nu\|^2_\delta=&\int_{\TT^2}\frac{(\nu(B(x,\delta(x))))^2}{(\delta(x))^4}dm \nonumber\\
\leq&\int_{\TT^2}\frac{1}{(\delta(x))^4}\left(\sum_{z\in A_\rho\cap B(x,2\delta(x))}\nu(B(z,\rho/4))\right)^2dm(x) \nonumber\\
\leq &\int_{\TT^2}\frac{|A_\rho\cap B(x,2\delta(x))|}{(\delta(x))^4}\left(\sum_{z\in A_\rho\cap B(x,2\delta(x))}(\nu(B(z,\rho/4)))^2\right)dm(x) \nonumber\\
=&\sum_{z\in A_\rho}(\nu(B(z,\rho/4)))^2\int_{\{x\in \TT^2| d(x,z)<2\delta(x)\}}\frac{|A_\rho\cap B(x,2\delta(x))|}{(\delta(x))^4}dm(x).
\end{align}

Since $A_\rho$ is $(\rho/5)$-separated, $\{B(z,\rho/11)\}_{z\in A_\rho\cap B(x, 2\delta(x))}$ is a collection of pairwise disjoint subsets of $B(x,3\delta(x))$. Therefore $|A_\rho\cap B(x,2\delta(x))|\cdot(\pi\rho^2/121)\leq 9\pi(\delta(x))^2$. Hence \eqref{delta<rho-1} implies that
\begin{align}\label{delta<rho-2}
\|\nu\|^2_\delta\leq &\sum_{z\in A_\rho}(\nu(B(z,\rho/4)))^2\int_{\{x\in M| d(x,z)<2\delta(x)\}}\frac{1}{(\delta(x))^4}\cdot \frac{9\cdot 121(\delta(x))^2}{\rho^2}dm(x) \nonumber\\
=&\sum_{z\in A_\rho}(\nu(B(z,\rho/4)))^2\int_{\{x\in M| d(x,z)<2\delta(x)\}}\frac{1089}{(\delta(x))^2\rho^2} dm(x) \nonumber\\
=&\sum_{z\in A_\rho}(\nu(B(z,\rho/4)))^2\int_{B(z,\delta_-)}\frac{1089}{(\delta(x))^2\rho^2} dm(x) \nonumber\\
&+\sum_{z\in A_\rho}(\nu(B(z,\rho/4)))^2\int_{\{x\in M| d(x,z)<2\delta(x)\}\setminus B(z,\delta_-)}\frac{1089}{(\delta(x))^2\rho^2} dm(x) \nonumber\\
\leq &\sum_{z\in A_\rho}(\nu(B(z,\rho/4)))^2\int_{B(z,\delta_-)}\frac{1089}{\delta_-^2\rho^2} dm(x) \nonumber\\
&+\sum_{z\in A_\rho}(\nu(B(z,\rho/4)))^2\int_{\{x\in M| d(x,z)<2\delta(x)\}\setminus B(z,\delta_-)}\frac{1089}{(d(x,z)/2)^2\rho^2} dm(x) \nonumber\\
\leq &\sum_{z\in A_\rho}(\nu(B(z,\rho/4)))^2\left(\frac{1089\pi\C{C}{0}}{\rho^2}+\int_{B(z,2\delta_+)\setminus B(z,\delta_-)}\frac{1089\C{C}{0}}{(d(x,z)/2)^2\rho^2} d\Leb(x)\right)\nonumber\\
=&\sum_{z\in A_\rho}(\nu(B(z,\rho/4)))^2\left(\frac{1089\pi\C{C}{0}}{\rho^2}+2\pi\int_{\delta_-}^{2\delta_+}\frac{4356\C{C}{0}}{ r^2\rho^2}rdr\right)\nonumber\\
=&\C{C}{0}\pi\left(\frac{1089}{\rho^2}+\frac{8712}{\rho^2}\log(2\delta_+/\delta_-)\right)\sum_{z\in A_\rho}(\nu(B(z,\rho/4)))^2.
\end{align}
Choose $\C{C}{2}=4(\C{C}{0}N_0)^2\cdot(1089+8712(1+\log(2)))$ and the lemma follows directly from \eqref{rho-discretize} and \eqref{delta<rho-2}.
\end{proof}

 \section{Preparatory lemmas}\label{sect:prep.lems}
 In this section, we prove several lemmas that will appear in the proof of \Cref{thm.mainthm}.
\subsection{Standing assumptions and notation I}\label{subsect:setting}
From Section \ref{sect:prep.lems} to Section \ref{sec : abscont}, we fix a pair $(f,g)$ so that it satisfies \textbf{(C1)} to \textbf{(C3)}.

Here are some notations we will use:
\begin{enumerate}
\item Fix some $\CS{\theta}{0}>0$. We choose open neighborhoods $\widetilde{\mathcal{U}}_{f}$ and $\widetilde{\mathcal{U}}_{g}$ of $f$ and $g$, respectively, such that the following holds:
\begin{itemize}
\item $\min_{x\in \mathbb{T}^{2}}\left\{\sphericalangle(E,F): E\in \mathcal{C}^{s}_{x}, F\in\mathcal{C}^{u}_{x}\right\}>\C{\theta}{0}.$
\item $\widetilde{\cU}_f$ and $\widetilde{\cU}_g$ satisfies (9) in Section \ref{subsect:other.no.}.
\item For any $x\in \TT^2$, for any $\widetilde{f}\in\widetilde{\cU}_f$ and $\widetilde{g}\in\widetilde{\cU}_g$, we have $d(E^u_{\widetilde{f}}(x),E^u_{\widetilde{g}}(x))>\CS{\theta}{\Delta}$ for some $\theta_\Delta>0$. 
\end{itemize}
\item By (10) in Section \ref{subsect:other.no.}, for any $x\in\TT^2$, for any $n\in\ZZ_{\geq 0}$, for any $\omega\in\widetilde{\cU}^\ZZ$, for any lines $F\in\cC_x^u$ and $E\in \cC^s_x$, there exists $\CS{C'}{3}=\C{C'}{3}(\C{\theta}{0},\C{C''}{0})>2$ such that 
\begin{align}\label{eqn:Lip.est.Dfn.u}\|Df^n_\omega(x)\|\leq \C{C'}{3}\|Df^n_\omega(x)|_F\|\leq \C{C}{3}\lambda_{u,+}^n,
\end{align}
and
\begin{align}\label{eqn:Lip.est.Dfn.s}\|Df^{-n}_\omega(x)\|\leq \C{C'}{3}\|Df^{-n}_\omega(x)|_E\|\leq \C{C}{3}\lambda_{s,-}^{-n},
\end{align}
where $\CS{C}{3}=\C{C'}{3}\C{C''}{0}>2$.
\item By (10) in Section \ref{subsect:other.no.}, for any $x\in\TT^2$, for any $\omega\in\widetilde{\cU}^\ZZ$, for any $n\in\ZZ_{\geq 0}$, there exists some constant $\CS{C}{4}=\C{C}{4}(\C{\theta}{0},\C{C''}{0},\lambda_{s,+}/\lambda_{u,-})>1$ such that the following holds:
\begin{itemize}
\item For any lines $F_1,F_2$ in $\cC^u_x$, we have 
\begin{align}\label{eqn:cone.contraction.u}
\sphericalangle(Df^n_\omega(x)F_1,Df^n_\omega(x)F_2)\leq\frac{\C{C}{4}}{\pi}\left(\frac{\lambda_{s,+}}{\lambda_{u,-}}\right)^{n}\sphericalangle(F_1,F_2)\leq  \C{C}{4}\left(\frac{\lambda_{s,+}}{\lambda_{u,-}}\right)^{n}.
\end{align}
\item For any lines $E_1,E_2$ in $\cC^s_x$, we have 
\begin{align}\label{eqn:cone.contraction.s}
\sphericalangle(Df^{-n}_\omega(x)E_1,Df^{-n}_\omega(x)E_2)\leq\frac{\C{C}{4}}{\pi}\left(\frac{\lambda_{s,+}}{\lambda_{u,-}}\right)^{n}\sphericalangle(E_1,E_2)\leq \C{C}{4}\left(\frac{\lambda_{s,+}}{\lambda_{u,-}}\right)^{n}.
\end{align}
\end{itemize}
\item For simplicity, we let $\lambda_s= \lambda_{s,-}$.
\end{enumerate}

When we introduce new constants in the rest of this paper, we do not track their dependence on $\C{C}{3}$, $\C{C'}{3}$ and $\C{C}{4}$. We only track their dependence on $\C{\theta}{0}$ and $\C{\theta}{\Delta}$ in Proposition \ref{prop.holderness} and its proof.

\subsection{Determinant for large words}

\begin{lem}\label{lem.neardeterminant}
Fix $\varepsilon>0$, there exist $\CS{n}{0}=n_{0}(\varepsilon)>0$,  and $C^1$-neighborhoods of $f$ and $g$, $\mathcal{U}_f$ and $\mathcal{U}_g$, respectively,   with the following property:

  Let $\mathcal{U} = \mathcal{U}_f \cup \mathcal{U}_g$. {For any $\omega \in \mathcal{U}^{\mathbb{N}}$,  for any $x\in \mathbb{T}^2$, for all $n\geq \C{n}{0}$, for any line $F \subset \mathcal{C}^u_x$, and for any line $E \subset \left(Df^n_{\omega}(x)\right)^{-1} \mathcal{C}^s_{f^n_{\omega}(x)}$,  we have}
  \[
e^{-\varepsilon n} < \|Df^n_{\omega}(x)|_{F}\| \| Df^n_{\omega}(x)|_{E}\| < e^{\varepsilon n}.
\] 
\end{lem}

\begin{proof}

Let us first show Lemma \ref{lem.neardeterminant} for $\omega \in \{f,g\}^{\mathbb{N}}$ and then we will see that the estimates we obtain hold for any sequence of diffeomorphisms $C^1$-near $f$ or $g$.

Fix $\omega \in \{f,g\}^{\mathbb{N}}$,  $n>0$, $x\in \mathbb{T}^2$, and lines $F\in \mathcal{C}^u_x$,  and $E \in \left(Df^n_{\omega}(x)\right)^{-1} \mathcal{C}^s_{f^n_{\omega}(x)}$.  Write 
\[
F_n = Df^n_{\omega}(x)  F \textrm{ and } E_n = Df^n_{\omega}(x)  E.
\]

Let $\{v_0,  w_0\}$ be  two unit vectors such that $v_0$ generates  $F$, and $w_0$ generates  $E$. Consider $U_0: \mathbb{R}^2 \to T_x \mathbb{T}^2$ the linear map defined by $e_1 \mapsto v_0$ and $e_2 \mapsto w_0$, where $\{e_1,e_2\}$ is the canonical basis of $\mathbb{R}^2$.  {(See Section \ref{subsect:other.no.}.)} Let $v_0^\perp$ be the unit vector perpendicular to $F$ that points in the same direction of the projection of $w_0$ into $F^\perp$. Using the bases $\{e_1, e_2\}$ and $\{v_0, v_0^\perp\}$, the linear transformation $U_0$ is given by the matrix
\[
U_0 =
\begin{pmatrix}
1 & B_0\\
0 & \cos \alpha_0
\end{pmatrix},
\]
where $B_0 $ is a number and $\alpha_0$ is the angle between $E$ and $F^\perp$.

Let 
\[
v_n = \frac{Df^n_{\omega}(x) v_0}{\|Df^n_{\omega}(x) v_0\|} \textrm{  and  } w_n = \frac{Df^n_{\omega}(x) w_0}{\|Df^n_{\omega}(x) w_0\|}.
\]
 Let $L_n$ be the linear transformation defined by $e_1 \mapsto v_n$ and $e_2 \mapsto w_n$.  Let $v_n^\perp$ be the unit vector in $F_n^\perp$ that points in the same direction as the projection of $w_n$ into $F_n^\perp$. 

Using the bases $\{e_1, e_2\}$ and $\{v_n, v_n^\perp\}$, the linear transformation $U_n$ is given by 
\[
U_n =
\begin{pmatrix}
1 & B_n\\
0 & \cos \alpha_n
\end{pmatrix},
\]
where $B_n$ is some number and $\alpha_n$ is the angle between $E_n$ and $F_n^\perp$.  Since $E \subset (Df^n_\omega(x))^{-1} \mathcal{C}^s_{f^n_\omega(x)}$, the assumptions in Section \ref{subsect:setting}, we have $\pi/2>\pi/2-\theta_0>\max\{\alpha_0, \alpha_n\}$.

{Recall that $f$ and $g$ preserves the smooth measure $m$, by (7) in Section \ref{subsect:other.no.}, for any $\omega\in\{f,g\}^\ZZ$, any $x\in \TT^2$ and any $n$, we have 
\begin{align}\label{eqn.m.preserving.det.est}
\C{C}{0}^{-2}\leq|\det Df_\omega^n(x)|\leq \C{C}{0}^2.
\end{align}}
Consider $D_n:\mathbb{R}^2 \to \mathbb{R}^2$ given by $D_n = U_n^{-1} \circ Df^n_{\omega}(x) \circ U_0$. {By \eqref{eqn.m.preserving.det.est}, we have
\[
-\frac{2}{n}\log\C{C}{0}\leq\left|\frac{1}{n} \log |\det D_n| -\frac{1}{n} (\log |\det U_n^{-1}| + \log |\det U_0|)\right|\leq \frac{2}{n}\log\C{C}{0}.
\]}
By (1) in Section \ref{subsect:setting}, we have $\sin(\C{\theta}{0})\leq|\det U_0| = | \cos \alpha_0|\leq 1$ and $1\leq |\det U_n^{-1}| = (\cos \alpha_n)^{-1}\leq (\sin(\C{\theta}{0}))^{-1}$. Hence,
\[
\left| \frac{1}{n} |\det D_n|\right| \leq \frac{2}{n} \log \C{C}{0}-\frac{2}{n}\log(\sin(\C{\theta}{0})).
\]
In particular, given $\varepsilon>0$ there exists $\C{n}{0}=\C{n}{0}(\varepsilon)$ such that for any $n \geq \C{n}{0}$, 
\[
-\varepsilon < \frac{1}{n} \log |\det D_n | < \varepsilon.
\]

However,  using the basis $\{e_1, e_2\}$, we have
\[
D_n  = \begin{pmatrix} \| Df^n_{\omega}(x)|_F\| & 0 \\ 0& \| Df^n_\omega (x) |_E\|\end{pmatrix}.
\]
Hence,
\[
|\det D_n| = \| Df^n_{\omega}(x)|_F\| \| Df^n_\omega (x) |_E\|,
\]
and the result follows for $\omega \in \{f,g\}^{\mathbb{N}}$. Observe that $\C{n}{0}$ above can be taken uniformly, independent on the choice of $\omega$. Therefore, for small neighborhoods $\mathcal{U}_f$ and $\mathcal{U}_g$ of $f$ and $g$, respectively, for any $\omega \in \cU^{\mathbb{N}}$, the same estimate holds. 
\end{proof}

The next lemma can be seen as a type of bounded distortion.

\begin{lem}\label{lem.boundeddistortiondet}
Fix $\varepsilon>0$, there exist {$\CS{n}{1}=n_1(\varepsilon)\in \mathbb{N}$},  and $C^2$-neighborhoods of $f$ and $g$, $\mathcal{U}_f$ and $\mathcal{U}_g$, respectively,  with the following property:

Let $\mathcal{U} = \mathcal{U}_f \cup \mathcal{U}_g$. For {all $n\geq\C{n}{1}$, for any} $\omega \in \mathcal{U}^{\mathbb{N}}$,  $z\in \mathbb{T}^2$,  $\rho'\in(0,1)$,  any two points $x,y \in \mathbb{T}^2$ with $f^n_\omega(x), f^n_\omega(y) \in B(f^n_\omega(z), \lambda_s^n \rho')$, for any line $F \subset \mathcal{C}^u_x$, and for any line $E \subset \left(Df^n_{\omega}(y)\right)^{-1} \mathcal{C}^s_{f^n_{\omega}(y)}$, {we have,}
\[
e^{-2\varepsilon n} < \|Df^n_{\omega}(x)|_{F}\| \| Df^n_{\omega}(y)|_{E}\| < e^{2\varepsilon n}.
\] 
\end{lem}
\begin{proof}
Fix $\mathcal{U}$ a $C^2$-neighborhood small enough so that Lemma \ref{lem.neardeterminant} holds {for $\varepsilon$}, and fix $\omega \in \mathcal{U}^\mathbb{N}$.  Let $p\mapsto E^s_{\omega, p}$ be the stable field, which is well defined since the stable direction only depends on the future. Fix $z\in \mathbb{T}^2$, and since {$E^s_{\omega,z} \subset (Df_\omega^n(z))^{-1} \mathcal{C}^s_{f^n_\omega(z)}\subset\cC^s_z $} for every $n\in \mathbb{N}$, by Lemma \ref{lem.neardeterminant}, 
\[
e^{-\varepsilon n} < \|Df^n_{\omega}(z)|_{F^u_{z}}\| \| Df^n_{\omega}(z)|_{E^s_{\omega,z}}\| < e^{\varepsilon n},~\forall n\geq \C{n}{0}(\varepsilon).
\]
Suppose that $x,y$ verify the condition of the Lemma \ref{lem.boundeddistortiondet}.  Let $F_x \subset \mathcal{C}^u_x$ and {$E_y \subset (Df_\omega^n(y))^{-1}\mathcal{C}^s_{f^n_\omega(y)} \subset\cC^s_y $}.  Let us start by comparing $\| Df^n_{\omega}(y)|_{E_y}\|$ with $\| Df^n_{\omega}(z)|_{E^s_{\omega,z}}\|$. In what follows, we write {$y_j:=f^j_\omega(y)$, $ E_{y_j} :=Df^j_{\omega}(y) E_y$ and $z_j = f^j_\omega(z)$. Then we have 
\begin{align*}
\displaystyle  &\left|\log \| Df^n_{\omega}(y)|_{E_y}\| - \log \| Df^n_{\omega}(z)|_{E^s_{\omega,z}}\|\right|\\
\displaystyle \leq &\sum_{j=0}^{n-1}  \left|\log \| Df_{\sigma^j(\omega)}(y_j)|_{E_{y_j}}\| - \log \| Df_{\sigma^j(\omega)}(z_j)|_{E^s_{\sigma^j(\omega),z_j}}\|\right|  \\
\displaystyle \leq &\sum_{j=0}^{n-1}  \left|\log \| Df_{\sigma^j(\omega)}(y_j)|_{E_{y_j}}\| - \log \| Df_{\sigma^j(\omega)}(y_j)|_{E^s_{\sigma^j(\omega),y_j}}\|\right|  \\
\displaystyle &+\sum_{j=0}^{n-1} \left|\log \| Df_{\sigma^j(\omega)}(y_j)|_{E^s_{\sigma^j(\omega), y_j}}\| - \log \| Df_{\sigma^j(\omega)}(y_j)|_{E^s_{\sigma^j(\omega),z_j}}\|\right| \\
\displaystyle &+\sum_{j=0}^{n-1}\left|\log \| Df_{\sigma^j(\omega)}(y_j)|_{E^s_{\sigma^j(\omega), z_j}}\| - \log \| Df_{\sigma^j(\omega)}(z_j)|_{E^s_{\sigma^j(\omega),z_j}}\|\right| .
\end{align*}

Observe the following:

\begin{itemize}
\item By \eqref{eqn:Lip.est.Dfn.s} and the fact that $d(y_n,z_n)\leq \lambda^n_s\rho'$, we have $d(y_j, z_j) < \C{C}{3}\lambda_s^j \rho'$.
\item By Proposition \ref{prop:UH}, the stable bundle is $(\C{L}{0},\theta)$-H\"older continuous.
\item Let $\lambda={\lambda_{s,+}/\lambda_{u,-}}$. By \eqref{eqn:cone.contraction.s}, we have $\sphericalangle(E_{y_j}, E^s_{\sigma^j(\omega),y_j}) \leq \C{C}{4}\lambda^{n-j}$. (See the definition of $\C{\theta}{0}$.)
\end{itemize}
By (9) in Section \ref{subsect:other.no.}, there exists some constant $\CS{C}{6}>0$ depending only on $\C{C'}{0}$ such that
\begin{align*}
 |\log \| Df_{\sigma^j(\omega)}(y_j)|_{E_{y_j}}\| - \log \| Df_{\sigma^j(\omega)}(y_j)|_{E^s_{\sigma^j(\omega),y_j}}\||  \leq \C{C}{6}\sphericalangle(E_{y_j}, E^s_{\sigma^j(\omega),y_j}) \leq \C{C}{6} \C{C}{4} \lambda^{n-j},
\end{align*}
\begin{align*}
&|\log \| Df_{\sigma^j(\omega)}(y_j)|_{E^s_{\sigma^j(\omega),y_j}}\| - \log \| Df_{\sigma^j(\omega)}(y_j)|_{E^s_{\sigma^j(\omega),z_j}}\|| \\
\leq &\C{C}{6}d(E^s_{\sigma^j(\omega),y_j},E^s_{\sigma^j(\omega),z_j})\leq \C{C}{6}  \C{L}{0} d(y_j, z_j)^{\theta} \leq \C{C}{6} \C{L}{0} ( \C{C}{3}\lambda_s^j \rho')^\theta
\end{align*}
and
\begin{align*}
|\log \| Df_{\sigma^j(\omega)}(y_j)|_{E^s_{\sigma^j(\omega),z_j}}\| - \log \| Df_{\sigma^j(\omega)}(z_j)|_{E^s_{\sigma^j(\omega),z_j}}\|| \leq \C{C}{6}d(y_j,z_j)\leq \C{C}{6}  \C{C}{3}\lambda_s^j \rho'.
\end{align*}
}

Hence,
\[
\displaystyle  |\log \| Df^n_{\omega}(y)|_{E_y}\| - \log \| Df^n_{\omega}(z)|_{E^s_{\omega,z}}\|   \leq {\C{C}{6}\left(\sum_{j=0}^{n-1}  \C{C}{4} \lambda^{n-j} + \sum_{j=0}^\infty  (\C{L}{0}( \C{C}{3}\lambda_s^j \rho')^\theta +\C{C}{3}\lambda_s^j \rho'). \right)}
\]
Observe that 
\[
\displaystyle  \lim_{n\to +\infty} {\sum_{j=0}^{n-1}  \C{C}{4} \lambda^{n-j} + \sum_{j=0}^\infty  (\C{L}{0}( \C{C}{3}\lambda_s^j \rho')^\theta+\C{C}{3}\lambda_s^j \rho')} < +\infty.
\]
Therefore, there exists $L_s$ such that
\[
\displaystyle e^{-L_s} \leq \frac{\| Df^n_{\omega}(y)|_{E_y}\|}{ \| Df^n_{\omega}(z)|_{E^s_{\omega,z}}\|} \leq e^{L_s}.
\]
Fix an $(L,\theta)$-H\"older line field $p\mapsto F^u_p$ contained in $\mathcal{C}^u$. By a similar computation, using {$p\mapsto F^u_p$ and $F_x$ instead of the stable field and $E_y$}, one can find a constant $L_u$ such that
\[
\displaystyle e^{-L_u} \leq \frac{\|Df^n_\omega(x)|_{F_x}\|}{\|Df^n_\omega(z)|_{F^u_z}\|} \leq e^{L_u}.
\]
Therefore,
\begin{align*}
\displaystyle &\|Df^n_\omega(x)|_{F_x}\|\cdot \| Df^n_{\omega}(y)|_{E_y}\| \\
\displaystyle  =&\frac{\|Df^n_\omega(x)|_{F_x}\|}{\|Df^n_\omega(z)|_{F^u_z}\|}\cdot \frac{\| Df^n_{\omega}(y)|_{E_y}\|}{\| Df^n_{\omega}(z)|_{E^s_{\omega,z}}\|}\cdot \|Df^n_{\omega}(z)|_{F^u_{z}}\| \| Df^n_{\omega}(z)|_{E^s_{\omega,z}}\| \leq e^{L_s + L_u} e^{\varepsilon n}
\end{align*}
It suffices to take $\C{n}{1}$ large enough so that $\C{n}{1}\geq \C{n}{0}$ and $e^{L_s + L_u} < e^{\varepsilon \C{n}{1}}$. The lower bound follows from similar computations.
\end{proof}

\section{Admissible measures}

Let $f,g \in \mathrm{Diff}^2_m(\mathbb{T}^2)$ be two Anosov diffeomorphisms verifying conditions \textbf{(C1)} and \textbf{(C2)}. Fix $\widetilde{\mathcal{U}}_f$ and $\widetilde{\mathcal{U}}_g$ $C^2$-neighborhoods of $f$ and $g$ satisfying (9) in Section \ref{subsect:other.no.}. Let $\widetilde{\mathcal{U}} = \widetilde{\mathcal{U}}_f \cup \widetilde{\mathcal{U}}_g$. Since both $f$ and $g$ preserves $m$, we also assume that $\widetilde{\cU}$ is so small such that for any $\widetilde{f}\in\widetilde{\cU}$, we have
\begin{align}\label{eqn:quant.avp}
\left(\frac{1+\lambda_{u,-}}{2}\right)^{-1}<\frac{d(\widetilde{f}_*m)}{dm}<\frac{1+\lambda_{u,-}}{2}.
\end{align}

For any $C^2$-curve $\gamma:[a,b]\to\TT^2$ and any $t\in[a,b]$, we denote by 
\begin{align}\label{eqn:curv.notation}
\kappa(t;\gamma)=\det(\dot\gamma(t),\ddot\gamma(t))/\|\dot\gamma(t)\|^3
\end{align} 
the curvature of $\gamma$ at $\gamma(t)$. (See (5) in Section \ref{subsect:other.no.})

\begin{lem}\label{lem.curvaturecontrol}
There exist constants $K_0=\CS{K}{0}(\widetilde{\mathcal{U}})>0$ and $n_2=\CS{n}{2}(\widetilde{\mathcal{U}})\in \mathbb{N}$ such that if $\gamma$ is a $C^2$-curve tangent to $\mathcal{C}^u$ such that $|\kappa(\cdot;\gamma)|\leq \C{K}{0}$, then for any $\omega \in \widetilde{\mathcal{U}}^{\mathbb{N}}$, and any $n\geq \C{n}{2}$, we have $|\kappa(\cdot;f^n_\omega(\gamma))|\leq \C{K}{0}$.
\end{lem}
\begin{proof}
Let $\omega \in \widetilde{\mathcal{U}}^{\mathbb{N}}$. Suppose that $\gamma:[0,a] \to \mathbb{T}^2$ is parametrized by arclength with curvature bounded from above by $\C{K}{0}$, we will find later what $\C{K}{0}$ must be.  

{
Let $\gamma_{n,\omega}(t) = f^n_\omega(\gamma(t))$. Observe that 
\[
\dot{\gamma}_{n,\omega}(t) = Df^n_\omega(\gamma(t))\dot{\gamma}(t) \textrm{ and } \ddot{\gamma}_{n,\omega}(t) = Df^n_\omega(\gamma(t)) \ddot{\gamma}(t) + D^2f^n_\omega(\gamma(t))(\dot{\gamma}(t), \dot{\gamma}(t)).
\]
Hence,
\begin{align}\label{eqn:curv.after.n}
&\displaystyle |\kappa(t;\gamma_{n,\omega})| \nonumber\\
=& \frac{|\det(\dot{\gamma}_{n,\omega}(t), \ddot{\gamma}_{n,\omega}(t))|}{\|\dot{\gamma}_{n,\omega}(t)\|^3} \nonumber\\ 
\leq& \frac{|\det(\dot{\gamma}_{n,\omega}(t), Df^n_\omega(\gamma(t)) \ddot{\gamma}(t))|}{\|\dot{\gamma}_{n,\omega}(t)\|^3}+ \frac{|\det(\dot{\gamma}_{n,\omega}(t), D^2f^n_\omega(\gamma(t))(\dot{\gamma}(t), \dot{\gamma}(t)))|}{\|\dot{\gamma}_{n,\omega}(t)\|^3}
\end{align}

Notice that for any $n\in\ZZ_+$, by (9) in Section \ref{subsect:other.no.}, there exists some constant $\CS{C}{7}=C_7(n,\C{C'}{0})$ such that for any $\omega'\in\widetilde{\cU}^\ZZ$, we have 
$$\|f^n_{\omega'}\|_{C^2}\leq\C{C}{7}(n,\C{C'}{0}).$$
In particular, by (10) in Section \ref{subsect:other.no.} and the above, we have the following estimate for the second term in \eqref{eqn:curv.after.n}.
\begin{align}\label{eqn:curv.after.n.1}
\frac{|\det(\dot{\gamma}_{n,\omega}(t), D^2f^n_\omega(\gamma(t))(\dot{\gamma}(t), \dot{\gamma}(t)))|}{\|\dot{\gamma}_{n,\omega}(t)\|^3}  \leq& \frac{\|\dot{\gamma}_{n,\omega}(t)\|\|D^2f^n_\omega(\gamma(t))\|}{\|\dot{\gamma}_{n,\omega}(t)\|^3}  \nonumber\\
=& \frac{\C{C}{7}(n,\C{C'}{0})}{\|Df^n_\omega(\gamma(t)) \dot{\gamma}(t)\|^2} \leq  \frac{\C{C}{7}(n,\C{C'}{0})(\C{C''}{0})^2}{\lambda_{u,-}^{2n}}.
\end{align}

Recall that $\gamma$ is parametrized by arclength. In particular, $\ddot{\gamma}(t)$ is perpendicular to $\dot{\gamma}(t)$ for every $t\in [0,a]$. Moreover, $\|\ddot{\gamma}(t)\|=|\kappa(t;\gamma)|$. Thus, by \eqref{eqn:quant.avp}, (7) and (10) in Section \ref{subsect:other.no.}, we have the following estimate for the first term in \eqref{eqn:curv.after.n}.
\begin{align}\label{eqn:curv.after.n.2} 
\frac{|\det(\dot{\gamma}_{n,\omega}(t), Df^n_\omega(\gamma(t)) \ddot{\gamma}(t))|}{\|\dot{\gamma}_{n,\omega}(t)\|^3}\nonumber
&= \frac{|\det(Df^n_\omega(\gamma(t)) \dot{\gamma}(t), Df^n_\omega(\gamma(t)) \ddot{\gamma}(t))|}{\|Df^n_\omega(\gamma(t)) \dot{\gamma}(t)\|^3}\nonumber\\
&= \frac{|\kappa(t;\gamma)|\det(Df^n_\omega(\gamma(t)))}{\|Df^n_\omega(\gamma(t)) \dot{\gamma}(t)\|^3} 
\leq\frac{\displaystyle |\kappa(t;\gamma)|\C{C}{0}^2\left(1+\lambda_{u,-}\right)^n}{(\C{C''}{0})^{-3}\lambda_{u,-}^{3n}\cdot 2^n}\nonumber\\
&<\frac{\C{C}{0}^2(\C{C''}{0})^3}{\lambda_{u,-}^{2n}}\cdot|\kappa(t;\gamma)|.
\end{align}
Apply the above to \eqref{eqn:curv.after.n}, we have
\begin{align}\label{eqn:curv.LY}
|\kappa(t;\gamma_{n,\omega})| \leq \frac{\C{C}{0}^2(\C{C''}{0})^3}{\lambda_{u,-}^{2n}}\cdot|\kappa(t;\gamma)| + \frac{\C{C}{7}(n,\C{C'}{0})(\C{C''}{0})^2}{\lambda_{u,-}^{2n}}.
\end{align}
Choose $\CS{n'}{2}>0$ such that $\C{C}{0}^2(\C{C''}{0})^3\lambda_{u,-}^{-2\C{n'}{2}}<1/2$. For simplicity, we write 
$$\CS{K'}{0}=\frac{4\C{C}{7}(\C{n'}{2},\C{C'}{0})(\C{C''}{0})^2}{\lambda_{u,-}^{2\C{n'}{2}}}.$$
Choose
$$\C{K}{0}=\max\left\{\C{K'}{0}, \max_{1\leq m\leq \C{n'}{2}-1}\left\{\frac{\C{C}{0}^2(\C{C''}{0})^3}{\lambda_{u,-}^{2m}}\cdot \C{K'}{0}+ \frac{\C{C}{7}(m,\C{C'}{0})(\C{C''}{0})^2}{\lambda_{u,-}^{2n}}\right\}\right\}$$
and 
$\C{n}{2}\in\C{n'}{2}\ZZ_+$ such that $(1/2)^{\C{n}{2}/\C{n'}{2}}< \frac{\C{K'}{0}}{4\C{K}{0}}$. Then for any $n>\C{n}{2}$, we write $n=m\C{n'}{2}+q$ for some $m\in\ZZ_+$ and $q\in \{0,\cdots, \C{n'}{2}-1\}$. In particular, $m\C{n'}{2}\geq \C{n}{2}$ and hence $(1/2)^{m}< \frac{\C{K'}{0}}{4\C{K}{0}}$. Therefore, if $|\kappa(t;\gamma)|\leq \C{K}{0}$, \eqref{eqn:curv.LY} in the case $n=\C{n'}{2}$ implies that $|\kappa(t;\gamma_{m\C{n'}{2},\omega})|\leq \C{K'}{0}\leq \C{K}{0}$. If in addition that $q\neq 0$, then one can apply \eqref{eqn:curv.LY} in the case $n=q$ to $\gamma_{m\C{n'}{2},\omega}$ and show that $|\kappa(t;\gamma_{n,\omega})|\leq \C{K}{0}$. This finishes the proof.
}

\end{proof}

\begin{dfn}
A \emph{$\widetilde{\mathcal{U}}$-admissible curve} is a  $C^2$-curve tangent to $\mathcal{C}^u$ having curvature bounded from above by $\C{K}{0}(\widetilde{\mathcal{U}})$, where $\C{K}{0}(\widetilde{\mathcal{U}})$ is a constant as in Lemma \ref{lem.curvaturecontrol}. 
\end{dfn}

Let $\gamma$ be a {$\widetilde{\cU}$-admissible} curve and let $m_\gamma$ be the arc length measure on $\gamma$.  
\begin{dfn}\label{dfn:L-good} Given a constant $L>0$, we say that a measure $\nu_{\gamma}$ supported on $\gamma$ is \emph{$L$-good} if there exists a positive function $\rho $ such that $\log \rho$ is $L$-Lipschitz and $d \nu_{\gamma} (\cdot) = \rho(\cdot) d m_\gamma(\cdot)$.
\end{dfn}  Note that if $\nu_{\gamma}$ is $L$-good then $\nu_{\gamma}$ is $L'$-good for all $L'\ge L$. 


\begin{lem}\label{lem.goodmeasurespreserved}
There exists $\CS{L}{1}(\widetilde{\mathcal{U}})$ such that, for each $L\ge \C{L}{1}$, there is $n_3=\CS{n}{3}(\widetilde{\mathcal{U}},L)\geq\C{n}{2}(\widetilde{\mathcal{U}})$ such that, for any $\widetilde{\mathcal{U}}$-admissible curve $\gamma$, for any $L$-good measure $\nu_{\gamma}$ on $\gamma$, for all $\omega \in \widetilde{\mathcal{U}}^{\mathbb{N}}$, and for any $n\geq \C{n}{3}$, the measure $(f_{\omega}^n)_* \nu_\gamma$ is $L$-good.
\end{lem}

\begin{proof}
Fix $\omega \in \widetilde{\mathcal{U}}^{\mathbb{N}}$ and let $\gamma$ be an admissible curve.  By Lemma \ref{lem.curvaturecontrol}, for any $n\in \mathbb{N}$, the curve $f^n_\omega (\gamma)$ is a $C^2$-curve with uniformly bounded curvature. 

For each $n\in \mathbb{N}$, and $y\in f^n_\omega(\gamma)$, let {$J_{\omega,n}(y) := \|(Df^n_\omega((f^n_\omega)^{-1}(y))^{-1}|_{T_y f^n_\omega(\gamma)}\|$.}  By the change of variables formula, for any measurable set $A$, we have
\[
(f^n_\omega)_*\nu_\gamma (A) = \displaystyle \int_{(f^n_\omega)^{-1}(A) \cap \gamma} \rho(x) dm_\gamma(x) = \int_{A\cap f^n_\omega(\gamma)} \rho((f^n_\omega)^{-1}(y)) J_{\omega,n}(y) dm_{f^n_\omega(\gamma)}(y).
\]
Hence, the density of $(f^n_\omega)_*\nu_\gamma$ with respect to $m_{f^n_\omega(\gamma)}$ is given by 
\[
\rho_n(y) = \rho((f^n_\omega)^{-1}(y))J_{\omega,n}(y).
\]  
For any $y_1,y_2\in f^n_\omega(\gamma)$, we have 
\begin{align}\label{eqn:adm.msr}
\begin{split}
|\log \rho_n(y_1) - \log \rho_n(y_2)| \leq& |\log \rho((f^n_\omega)^{-1}(y_1) - \log \rho((f^n_\omega)^{-1}(y_2))|\\
 & +|\log J_{\omega,n}(y_1) - \log J_{\omega,n}(y_2)|.
\end{split}
\end{align}
By (10) in Section \ref{subsect:other.no.}, the fact that $\gamma$ is $\widetilde{\cU}$-admissible and the fact that $\log\rho$ is $L$-Lipschitz, we have
\begin{align}\label{eqn:adm.msr.est1}
\begin{split}
|\log \rho((f^n_\omega)^{-1}(y_1) - \log \rho((f^n_\omega)^{-1}(y_2))|\leq& Ld_\gamma((f^n_\omega)^{-1}(y_1),(f^n_\omega)^{-1}(y_2))\\
\leq& L\C{C''}{0}\lambda^{-n}_{u,-}d_{f^n_\omega\gamma}(y_1,y_2).
\end{split}
\end{align}
Before estimating the second term in \eqref{eqn:adm.msr}, we observe that for any $K>0$ and any $C^2$-curve $\gamma$ on $\TT^2$ satisfies $|\kappa(\cdot,\gamma)|\leq K$, the following holds:
\begin{enumerate}
\item  For any $p_1,p_2\in\gamma$, we have $d(T_{p_1}\gamma, T_{p_2}\gamma)\leq \sqrt{1+K^2}\cdot d_\gamma(p_1,p_2)$. (Here, we view $T_{p_1}\gamma$ and $T_{p_1}\gamma$ as points in $\PP T\TT^2$. See (6) in Section \ref{subsect:other.no.}.)
\item By similar computations in \eqref{eqn:curv.after.n}, \eqref{eqn:curv.after.n.1} and \eqref{eqn:curv.after.n.2}, for any $C^2$-map $F:\TT^2\to\TT^2$ with $\|F\|_{C^2}<C$, and for any $n>0$, there exists some constant $C_8=\CS{C}{8}(n,C,K)>0$ such that $\sum_{j=0}^{n-1}\sup_{t}|\kappa(t,F^j(\gamma)|\leq \C{C}{8}(n,C,K)$.
\end{enumerate}
To simplify notations, we let $y_j^{i} = f^{n-i}_\omega ((f^n_\omega)^{-1}(y_j))$, for $j=1,2$ and 
$$\CS{C'}{8}=C'_8(\widetilde{\cU}):=(2\C{C'}{0})^2\sqrt{1+\left(\max\{\C{K}{0}(\widetilde{\cU}),\C{C}{8}(\C{n}{2}(\widetilde{\cU}),2\C{C'}{0},\C{K}{0}(\widetilde{\cU}))\}\right)^2}.$$ By the above discussion (used in the third and the fourth inequalities below), Lemma \ref{lem.curvaturecontrol} (used in the fourth inequalities below) and (9) and (10) in Section \ref{subsect:other.no.} (used in the second and the fifth inequalities below), we have
\begin{align}\label{eqn:adm.msr.est2}
 &|\log J_{\omega,n}(y_1) - \log J_{\omega,n}(y_2)| \nonumber\\
 \leq & \sum_{i=0}^{n-1} |\log J_{\sigma^{n-i - 1}(\omega), 1}(y^i_1) - \log J_{\sigma^{n-i-1}(\omega), 1}(y_2^i)|\nonumber\\
 \leq & (2\C{C'}{0})^2\sum_{i=0}^{n-1}  d(T_{y^i_1}f^{n-i}_\omega(\gamma), T_{y^i_2}f^{n-i}_\omega(\gamma)) \nonumber \\
 \leq &(2\C{C'}{0})^2\sqrt{1+(\sup_t|\kappa(t,f^{n-i}_\omega(\gamma))|)^2}\cdot\sum_{i=0}^{n-1}d_{f^{n-i}_\omega(\gamma)}(y^i_1, y^i_2) \nonumber \\
  \leq &\C{C'}{8}\sum_{i=0}^{n-1}d_{f^{n-i}_\omega(\gamma)}(y^i_1, y^i_2) \nonumber \\
 \leq & \C{C'}{8}\sum_{i=0}^{n-1}\C{C''}{0} \lambda_{u,-}^{-i} d_{f^n_\omega(\gamma)}(y_1,y_2) \leq \frac{\C{C'}{8}\C{C''}{0}}{1-\lambda_{u,-}^{-1}}d_{f^n_\omega(\gamma)}(y_1,y_2).
\end{align}
Apply \eqref{eqn:adm.msr.est1} and \eqref{eqn:adm.msr.est2} to \eqref{eqn:adm.msr}, the lemma then follows from choosing
$\C{L}{1}:=\frac{2\C{C'}{8}\C{C''}{0}}{1-\lambda_{u,-}^{-1}}$ and $\C{n}{3}\geq\C{n}{2}(\widetilde{\cU})$ such that $\C{C''}{0}\lambda_{u,-}^{-\C{n}{3}}<1/2$.

\end{proof}
Let $\mathfrak{C}(\mathcal{U}',L')$ be the set of $L'$-good measures with respect to $\mathcal{U}'$ for each $L'>0$ and an open set $\mathcal{U}'$ containing $f$ and $g$.
We could consider $\mathfrak{C}(\mathcal{U}',L')$ as a measurable subset of $M(\mathbb{T}^{2})$ where $M(\mathbb{T}^{2})$ is the set of all finite measures on $\mathbb{T}^{2}$. Here, we put $M(\mathbb{T}^{2})$ with weak Borel structure, that is, the smallest $\sigma$-algebra that makes the map $\delta\mapsto \delta(E)$ becomes measurable for all finite measure $\delta\in M(\mathbb{T}^{2})$ and for all Borel set $E\subset \mathbb{T}^{2}$, so that $M(\mathbb{T}^{2})$ becomes a standard Borel space.

\begin{dfn} \label{def-admissiblemeasure}
We say that a measure $\nu_{0}$ on $\mathbb{T}^{2}$ is \emph{$(\widetilde{\mathcal{U}},L)$-admissible} if there exists a measure $\widehat{\nu}_{0}$ on $\mathfrak{C}(\widetilde{\mathcal{U}},L)$, such that $\displaystyle \nu_{0} = \int_{\mathfrak{C}(\widetilde{\mathcal{U}},L)} \widetilde{\nu}_{0} d\widehat{\nu}_{0}(\widetilde{\nu}_{0}).$
\end{dfn}
\begin{dfn} For each $\widetilde{\mathcal{U}}$ and $L'$, let  $\displaystyle \nu_{0} = \int_{\mathfrak{C}(\widetilde{\mathcal{U}},L')} \widetilde{\nu}_{0} d\widehat{\nu}_{0}(\widetilde{\nu}_{0})$ be a $(\widetilde{\mathcal{U}},L)$- admissible measure. We say that $\nu_{0}$ is \emph{supported on curves of length bounded from below by $r>0$} if for $\widehat{\nu}_{0}$-almost every $\widetilde{\nu}_{0}$, the measure $\widetilde{\nu}_{0}$ is supported on an admissible curve of length at least $r$. 
\end{dfn}

  The following corollary is a direct consequence of \Cref{lem.goodmeasurespreserved}.

\begin{cor}\label{cor.admissibletoadmissible}
{Let $\C{L}{1}$ and $\C{n}{3}$ be the same as in Lemma \ref{lem.goodmeasurespreserved}. For all sufficiently small open neighborhoods $\widetilde{\mathcal{U}}_{f}$ and $\widetilde{\mathcal{U}}_{g}$ with $\widetilde{\mathcal{U}}=\widetilde{\mathcal{U}}_{f}\cup \widetilde{\mathcal{U}}_{g}$, for any $L\geq \C{L}{1}(\widetilde{\cU})$, for any $(\widetilde{\mathcal{U}},L)$-admissible measure $\nu_0$, for any $\omega \in \widetilde{\mathcal{U}}^{\mathbb{N}}$ and for any $n\geq \C{n}{3}(\widetilde{\cU},L)$, the measure $(f^n_\omega)_* \nu_0$ is also $(\widetilde{\mathcal{U}},L)$-admissible.}
\end{cor} 

\section{H\"older regularity of measures on the projective bundle}\label{sec:holder}

In order to say that we have enough transversality for unstable manifolds, we need Proposition \ref{prop.holderness} below. Roughly, it says that, in \Cref{subsect:setting}, unstable directions cannot be concentrated too much in one direction.

\begin{prop}\label{prop.holderness}
Fix $\beta \in ( 0, \frac{1}{2}]$, and let $f$ and $g$ be diffeomorphisms as in the statement of Theorem \ref{thm.mainthm}. Then, there exist $\eta{=\eta(\beta, \C{\theta}{0},\C{\theta}{\Delta}})\in (0,1)$, neighborhoods $\mathcal{U}_f$ and $\mathcal{U}_g$ of $f$ and $g$, respectively, and constants $\CS{C}{5} =\C{C}{5} (\beta, \C{\theta}{0},\C{\theta}{\Delta})$ and $\alpha = \alpha(\beta, {\C{\theta}{0},\C{\theta}{\Delta}})$, with the following property:

For any probability measure $\mu$ on $\Diff^2(\mathbb{T}^2)$ such that $\mu(\mathcal{U}_\star) \in [\beta, 1-\beta]$, $\star=f,g$, for any $\widehat{\nu} = \{\widehat{\nu}_x\}_{x\in \mathbb{T}^{2}}$ continuous family of probability measures {$\widehat{\nu}_{x}\in \Prob(\PP T_x\TT^2)$ supported in $\PP\cC^u_x$}, for any $n> 0$, for any $x\in \mathbb{T}^2$, for any $u\in {\mathbb{P}T_x\mathbb{T}^{2}}$ and for any $r\geq \eta^n$, we have 
\[
(\mu^{*n}*\widehat{\nu})_x( B_r(u)) \leq \C{C}{5} r^{\alpha}.
\]

Here $B_{r}(u)$ is the open ball of radius $r$ centered at $u$ in ${\mathbb{P}T_x\mathbb{T}^{2}}$.
\end{prop}
\Cref{prop.holderness} is a direct corollary of Proposition \ref{prop.gen holderness} below.  \Cref{prop.gen holderness} gives a quantitative Holder regularity of fiberwise measure for certain Lipschitz homeomorphisms which behave similarly to Anosov diffeomorphisms satisfying the cone condition.

Let $X$, $Y$ be compact metric spaces. We denote by $\cD_X(X\times Y)$ the collection of Lipschitz homeomorphisms $F:X\times Y\to X\times Y$ such that the following holds:
\begin{itemize}
\item There exist a Lipschitz homeomorphism $T_F:X\to X$ satisfying $\pr_X\circ F=T_F\circ \pr_X$, where $\pr_X:X\times Y\to X$ is the natural projection map.
\item For any $x\in X$, the map $F_x:Y\to Y$ defined as $F_x(y)=\pr_Y(F(x,y))$ is a Lipschitz homeomorphism.
\end{itemize}
One can easily check that for any $F,F'\in\cD_X(X\times Y)$, $F\circ F'\in\cD_X(X\times Y)$. Similar to the notations after Definition \ref{dfn.stationary}, we let $\Omega_{n}(X\times Y)=\cD_X(X\times Y)^n$ for any $n\in\ZZ_+$. For any $\omega\in\Omega_{n}(X\times Y)$ with $\omega=(F_1,\cdots,F_n)$, we write $F^j_\omega=F_j\circ\cdots\circ F_1$ for any $j\in\{1,\cdots,n\}$. We also write $T_{F^{n}_{\omega}}=T_{F_{n}}\circ\cdots \circ T_{F_{1}}$.

For any $\cF\subset \cD_X(X\times Y)$ and any $\omega=(F_1,\cdots,F_n)\in\Omega_{n}(X\times Y)$ (or any $\omega=(F_1,F_2,\cdots)\in\Omega^+(X\times Y)$), we say that $\omega$ is an \emph{$\cF$-word} if $F_1,\cdots, F_n\in\cF$ (or $F_1,F_2,\cdots \in\cF$). 
\begin{dfn}\label{dfn.assump for holder}
Let $\mu\in\Prob(\cD_X(X\times Y))$. We introduce the following properties for $\mu$.
\begin{enumerate}
\item (\textbf{$(C,\lambda)$-unstable cone condition}) We say that $\mu$ satisfies the \emph{unstable cone condition} if there exist an open subset $\cO\subset X\times Y$ such that for any $x\in X$, $\cO_x:=\cO\cap\{x\}\times Y$ is an non-empty open subset of $\{x\}\times Y$. Moreover, for any $F\in\supp(\mu)$,
$$F_x(\pr_Y(\cO_x))\subset \pr_Y(\cO_{T_F(x)})~\mathrm{and}~\Lip(F^n_x|_{\pr_Y(\cO_x)})\leq C\lambda^n,$$
for some constants $C>0$ and $\lambda\in(0,1)$ which are independent of the choice of $x$ and $F$. $\cO$ is called the \emph{unstable cone bundle} for $\mu$.
\item (\textbf{$(k_0,\beta,\varphi)$-unstable separation condition}) 
We say that $\mu$ satisfies the \emph{$(k_{0},\beta,\varphi)$-unstable separation condition} if there exists some $k_0>0$, $\beta\in (0,1/2)$, $\varphi>0$ and two disjoint, $\mu$-measurable subsets $\cF_1,\cF_2\subset\cD_X(X\times Y)$, such that the following holds
\begin{itemize}
\item $\mu(\cF_1),\mu(\cF_2)\geq \beta$.
\item For any $x\in X$ and any $\cF_j$-word $\omega_j\in\Omega_{k_0}(X\times Y)$, $j=1,2$, we have
$$d_Y\left((F^{k_0}_{\omega_1})_{\left(T_{F^{k_0}_{\omega_1}}\right)^{-1}(x)}\left(\cO_{\left(T_{F^{k_0}_{\omega_1}}\right)^{-1}(x)}\right),(F^{k_0}_{\omega_2})_{\left(T_{F^{k_0}_{\omega_2}}\right)^{-1}(x)}\left(\cO_{\left(T_{F^{k_0}_{\omega_2}}\right)^{-1}(x)}\right)\right)>\varphi.$$ 
\end{itemize} 
\end{enumerate}
\end{dfn}
\begin{rmk}
One can easily check that if $\mu$ satisfies the \textbf{$(k_0,\beta,\varphi)$-unstable cone condition}, then for any integer $k\geq k_0$, $\mu$ also satisfies the \textbf{$(k,\beta,\varphi)$-unstable cone condition}.
\end{rmk}
For any $\nu\in\Prob(X\times Y)$, let $\{\nu_x\}_{x\in X}$ be the conditional measures with respect to the measurable partition $\{\{x\}\times Y\}_{x\in X}$ of $X\times Y$. By identifying $\{x\}\times Y$ with $Y$ via the natural map $(x,y)\to y$, we can view $\nu_x$ as probability measures on $Y$.

\begin{prop}\label{prop.gen holderness}
Let $\mu\in\mathrm{Prob}(\cD_X(X\times Y))$ be a probability measure satisfying the \textbf{$(C,\lambda)$-unstable cone condition} with unstable cone $\cO$, and the \textbf{$(k_0,\beta,\varphi)$-unstable separation condition}. We further assume that there exists a constant $L>1$ such that for any $x\in X$ and any $F\in\supp(\mu)$, we have
\begin{align}\label{eqn.inverseLip}
\mathrm{Lip}((F_x)^{-1})\leq L.
\end{align}
Then there exist some constants $C'>0$, $0<\kappa<1$ and $\gamma>0$ depending only on $C,\lambda,k_0,\beta,\varphi$ and $L$ such that for any $x\in M$, for any $\nu\in\Prob(X\times Y)$ supported on $\cO$, for any $n>0$, for any $y\in Y$ and for any $r>\kappa^n$, we have $\mu^{*n}*\nu$ is a probability measure supported on $\cO$ satisfying
$$(\mu^{*n}*\nu)_x(B_Y(y,r))\leq C'r^\gamma,$$
where $B_Y(y,r)$ denotes the open ball of radius $r$ centered at $y$ in $Y$.
\end{prop}
\begin{proof}
Since
\begin{align}\label{eqn.convolutioncond}
(\mu*\nu)_x=\int_{\cD_X(X\times Y)}\left(F_{(T_F)^{-1}(x)}\right)_*\nu_{(T_F)^{-1}(x)}d\mu(F).
\end{align}
The fact that $\mu^{* n}*\nu$ is supported on $\cO$ follows directly from the \textbf{unstable cone condition} and the assumption on $\nu$. By \eqref{eqn.inverseLip}, for any $x\in X$, for any $y\in Y$, for any $F\in\supp(\mu)$ and for any $\rho>0$, we have 
\begin{align}\label{eqn.radballpreimage}
F_x^{-1}(B_Y(y,\rho))\subset B_Y(F_x^{-1}(y),L\rho).
\end{align}
Let 
\begin{align}\label{eqn.supballvolume}
K(\rho,j):=\sup_{\substack{x\in X\\F\in\supp(\mu)}}(\mu^{*j}*\nu)_x(B_Y(y,\rho)),~\forall j\in\ZZ_+.
\end{align}
Choose a large positive integer $m_0=m_0(C,k_0,\lambda,\varphi)\geq 1$ such that $C\lambda^{m_0k_0}\cdot\mathrm{diam}(Y)<\varphi/4$. Then by the remark after Definition \ref{dfn.assump for holder}, the \textbf{$(C,\lambda)$-unstable cone condition} and the \textbf{$(k_0,\beta,\varphi)$-unstable separation condition} for $\mu$, there exist two disjoint, $\mu$-measurable subsets $\cF_1,\cF_2\subset\supp(\mu)$ with $\mu(\cF_1)\geq \beta$ and $\mu(\cF_2)\geq \beta$, such that for any $x\in X$ and any $y\in Y$, there exists some $j\in\{1,2\}$ such that for any $\cF_j$-word $\omega_j\in\Sigma_{m_0k_0}(X\times Y)$, we have
$$\left(\left(F^{m_0k_0}_{\omega_j}\right)_{\left(T_{F^{m_0k_0}_{\omega_j}}\right)^{-1}(x)}\left(\cO_{\left(T_{F^{k_0}_{\omega_j}}\right)^{-1}(x)}\right)\right)\cap B_Y(y,\varphi/4)=\emptyset.$$
As a corollary of \eqref{eqn.convolutioncond}, \eqref{eqn.radballpreimage} and the above, we have 
\begin{align}\label{eqn.gen holderness-1}
K(\rho,j)\leq \left(1-\beta^{m_0k_0}\right)K(L^{m_0k_0}\rho,j-m_0k_0),~\forall j\geq m_0k_0~\mathrm{and}~\forall \rho<\frac{\varphi}{4}.
\end{align}
Choose $\kappa=L^{-1}$. For any $r>\kappa^n$, we let
$$q(n,r):=\max\left\{\min\left\{\left[\frac{\log(\varphi/4r)}{k_0m_0\log(L)}\right],\left[\frac{n}{k_0m_0}\right]\right\},0\right\}.$$
Then for any $r\in(\kappa^n,1]$, we have
\begin{align}\label{eqn.gen holderness-2}
\left|q(n,r)-\frac{\log(1/r)}{k_0m_0\log(L)}\right|\leq1+\left|\frac{\log(\varphi/4)}{k_0m_0\log(L)}\right|.
\end{align}
We further choose
$$\gamma=-\frac{\log(1-\beta^{k_0m_0})}{k_0m_0\log(L)}>0$$
and 
$$C'=(1-\beta^{k_0m_0})^{-1-\left|\frac{\log(\varphi/4)}{k_0m_0\log(L)}\right|}\geq 1.$$ 
Since $K(\rho,j)\leq 1$ for any $\rho>0$ and any $j\in\ZZ_{\geq 0}$, $K(r,n)\leq C'r^\gamma$ obviously holds when $r>1$. When $r\in(\kappa^n,1]$, \eqref{eqn.gen holderness-1} and \eqref{eqn.gen holderness-2} imply that
\begin{align*}
K(r,n)\leq &(1-\beta^{k_0m_0})^{q(n,r)}K(L^{q(n,r)k_0m_0}r,n-q(n,r)k_0m_0) \\
\leq& (1-\beta^{k_0m_0})^{q(n,r)}\\
\leq& C'(1-\beta^{k_0m_0})^{\gamma\cdot\frac{\log(r)}{\log(1-\beta^{k_0m_0})}}=C'r^\gamma.
\end{align*}
Following the definition in \eqref{eqn.supballvolume}, the proof is complete.
\end{proof}
{\begin{proof}[Proof of Proposition \ref{prop.holderness}]
Let $X=\TT^2$ and $Y=\PP\RR^2$. We can naturally identify $X\times Y$ with $\PP T\TT^2$. 

Consider maps of the form $D\widetilde{f}:\PP T\TT^2\to \PP T\TT^2$ with $\widetilde{f}\in\cU_f\cup\cU_g$. Choose $\cO=\bigcup_{x\in\TT^2}\cC^u_x$. By \eqref{eqn:cone.contraction.u}, such maps satisfy the \textbf{$(C,\lambda)$-unstable cone condition} with $C=\C{C}{4}$ and $\lambda={\lambda_{s,+}/\lambda_{u,-}}$. (See Section \ref{subsect:setting}.) Let $\iota:\mathrm{Diff}^2(\TT^2)\to\cD_X(X\times Y)$ such that $\iota(F)=DF$. Then $\iota_*\mu$ satisfies the \textbf{$(k_0,\beta,\theta)$-unstable separation condition} for some $k_0=k_0(C,\lambda,\C{\theta}{\Delta})$ and $\theta=\C{\theta}{\Delta}/2$. ($\beta$ is given in the statement of Proposition \ref{prop.holderness}. $k_0$ is an integer such that $2\pi C\lambda^{k_0}<\C{\theta}{\Delta}/5$. To verify the \textbf{$(k_0,\beta,\theta)$-unstable separation condition} for $\iota_*\mu$, we choose $\cF_1=\iota({\cU}_f)$ and $\cF_2=\iota({\cU}_g)$. The rest follows from (1) in Section \ref{subsect:setting}.) Proposition \ref{prop.holderness} then follows from Proposition \ref{prop.gen holderness} with $L=2\C{C'}{0}$.
\end{proof}}

\section{Absolute continuity of stationary SRB measures}\label{sec : abscont}
\subsection{Standing assumptions and notation II}\label{subsect:setting2}
We retain the setting {in Section \ref{subsect:other.no.} and in} \Cref{subsect:setting}. 

\begin{enumerate}
\item Fix $\beta \in (0,\frac{1}{2}]$. 
\item Let  $\alpha=\alpha(\beta,\C{\theta}{0},\C{\theta}{\Delta})$, $\eta = \eta(\beta,\C{\theta}{0},\C{\theta}{\Delta})\in (0,1)$ and $\C{C}{5} = \C{C}{5}(\beta,\C{\theta}{0},\C{\theta}{\Delta})$ be the same as in Proposition \ref{prop.holderness}. 
\item Fix a positive constant $\varepsilon$ such that
$$0<\varepsilon<\min\left\{1,\frac{1+\lambda_{u,-}}{2},\frac{-\alpha \log \eta }{8}, \frac{-\alpha \theta\log (\lambda_s)}{10},\frac{-\alpha\log(\lambda_{s,+}/\lambda_{u,-})}{10}\right\},$$  
where  $\theta$ is the same as in Proposition \ref{prop:UH}.
\item Take open neighborhoods $\mathcal{U}_f$ and $\mathcal{U}_g$ no larger than the open neighborhoods in \Cref{prop.holderness} so that Lemma \ref{lem.neardeterminant} and Lemma \ref{lem.boundeddistortiondet} hold for $\varepsilon$. Moreover, we assume that for any $\widetilde{f}\in\cU_f\cup\cU_g$, we have $e^{-\varepsilon}<{d(\widetilde{f}_*m)}/{dm}<e^{\varepsilon}$ and $e^{-\varepsilon}<{d(\widetilde{f}^{-1}_*m)}/{dm}<e^{\varepsilon}.$ In particular, \eqref{eqn:quant.avp} holds. Let $\mathcal{U}=\mathcal{U}_{f}\cup \mathcal{U}_{g}$. Denote $\Sigma=\mathcal{U}^{\mathbb{Z}}$ and  $\Sigma^+ = \mathcal{U}^{\mathbb{N}}$.  
\item Take $\mu$ a probability measure such that $\mu(\mathcal{U}) = 1$ and $\mu(\mathcal{U}_\star) \in [\beta, 1-\beta]$, for $\star = f,g$.  

\item {Fix $0<\CS{\rho}{0}=\rho_0(\cU)<\min\{\frac{1}{100},\frac{3}{4\C{K}{0}(\cU)},\frac{\sin(\C{\theta}{0}/2)}{10\C{C}{3}},\frac{1}{10\C{C'}{3}}\}$. To simplify certain proofs, we assume in addition that for any $p\in \TT^n$, there exist lines $E,F\in\RR^2$ such that for any $q\in B(p,\C{C}{3}\C{\rho}{0})$, after identifying $\RR^2$ with $T_q\TT^2$, we have $F\in\cC^u_q$ and $E\in\cC^s_q$. Moreover, for any $F'\in\cC^u_q$ and for any $E'\in\cC^s_q$, we have $\max\{\sphericalangle(F,F'),\sphericalangle(E,E')\}<(\pi-\C{\theta}{0})/2$. } 
\end{enumerate}

\subsection{A key estimate for admissible measures}

Let $\nu_{0}$ be an admissible measure (see Definition \ref{def-admissiblemeasure}), and let $\widehat{\nu}_{0}$ be the measure defining $\nu_{0}$, that is,
\[
\nu_{0} = \int_{\mathfrak{C}} \widetilde{\nu}_{0} d\widehat{\nu}_{0}(\widetilde{\nu}_{0}).
\]  
Recall that we say that $\nu_{0}$ is \emph{supported on curves of length bounded from below by $r>0$} if for $\widehat{\nu}_{0}$-almost every $\widetilde{\nu}_{0}$, the measure $\widetilde{\nu}_{0}$ is supported on an admissible curve of length at least $r$.

%


For $n\in \mathbb{N}$, $x\in \mathbb{T}^2$ and $\omega \in \mathcal{U}^{\mathbb{Z}}$, we will use the following notation:
\begin{itemize}
\item $\displaystyle J^s_{\omega,n}(x) := \inf_{ E\in D(f^{n}_{\omega})^{-1}(x)\mathcal{C}^s_{x}}  \| Df^n_\omega((f^{n}_\omega)^{-1} (x))|_E\|$;
\item $\displaystyle  J^{u}_{\omega,n}(x):= \inf_{ F \in \mathcal{C}^u_{(f^{n}_\omega)^{-1}(x)}} \| Df^n_{\omega}((f^{n}_\omega)^{-1}(x))|_F\|$.
\end{itemize}
We also let $\CS{C}{9}=2\C{C'}{3}(\C{C''}{0})^{-1}$ for simplicity. Recall that $\lambda_s = \lambda_{s,-}$ and the constant $L_1 =\C{L}{1}(\cU)$ given by \Cref{lem.goodmeasurespreserved}. The main result in this subsection is given by the following lemma. 

 \begin{lem}\label{lemma.maininequality1}
For any $\rho'\in (0,\C{\rho}{0}(\cU))$ and any $L\geq \C{L}{1}(\cU)$, {there exists $\CS{n}{4}=\C{n}{4}(\varepsilon,L)\geq\C{n}{1}(\varepsilon)$ such that for any $n\geq\C{n}{4}$, for any $\rho \in (0, \lambda_s^n \rho')$, for any }$({\mathcal{U}},L)$-admissible measure $\nu_{0}$ supported on curves of length bounded from below by $2\C{C}{3}\lambda_s^{-n}\rho$, and for any $\omega \in \Sigma^+$, {we have} 
\[
\|(f^n_\omega)_*\nu_{0}\|^2_{ \rho} \leq  e^{6\varepsilon n} \|\nu_{0}\|^2_{\C{C}{9}\lambda_{s,+}^{-n} \rho }.
\]

 \end{lem}
\begin{proof}

Take $\rho' \in (0,{\C{\rho}{0}(\cU)})$ {and $\omega\in\Sigma^+$}.  For each $z\in \mathbb{T}^2$, write
$$
\displaystyle \widehat{J}^{u}_{\omega,n}(z):= \inf_{x\in B(z,\lambda_s^n \rho')} J^{u}_{\omega,n}(x) \textrm{ and } \displaystyle \widehat{J}^s_{\omega,n}(z) : =\inf_{x\in B(z,\lambda_s^n \rho')} J^s_{\omega,n}(x).
$$
By Lemma \ref{lem.boundeddistortiondet}, for any $n\geq \C{n}{1}(\varepsilon)$, we have
\begin{align}\label{eqn:maininequality1.1}
e^{-2\varepsilon n} \leq \widehat{J}^s_{\omega,n}(z) \widehat{J}^{u}_{\omega,n}(z)\leq e^{2\varepsilon n}.
\end{align}

Write $\nu_{0} = \int_{\mathfrak{C}} \widetilde{\nu}_{0} d\widehat{\nu}_{0}(\widetilde{\nu}_{0})$. Observe that $(f^n_\omega)_*\nu_{0} = \int_{{\FFC}} (f^{n}_\omega)_* \widetilde{\nu}_{0} d\widehat{\nu}_{0}(\widetilde{\nu}_{0})$. Therefore,
\[
\|(f^n_{\omega})_*\nu_{0}\|^2_{\rho}  = \frac{1}{\rho^{4}} \int_{\mathbb{T}^2} \left(\int_{{\FFC}} \widetilde{\nu}_{0}((f^{n}_{\omega})^{-1}(B(z, \rho)))\ d\widehat{\nu}_{0}(\widetilde{\nu}_{0})\right)^2 dm(z).
\]
{By the assumptions on $\nu_0$, for $\widehat{\nu}_{0}$-almost every $\widetilde{\nu}_{0}$, $\widetilde{\nu}_{0}$ is an $L$-good measure supported on an admissible curve $\gamma_{\widetilde{\nu}_0}$ with length bounded from below by $2 \C{C}{3}\lambda_s^{-n} \rho$. (See Definition \ref{dfn:L-good} and Definition \ref{def-admissiblemeasure}.) For such a $\widetilde{\nu}_{0}$,} let us estimate $\widetilde{\nu}_{0}((f^{n}_{\omega})^{-1}(B(z,\rho)))$. {Let $m_{\widetilde{\nu}_{0}}$ be the arclength measure on $\gamma_{\widetilde{\nu}_0}$. Let $\zeta_{\widetilde{\nu}_{0}}=d\widetilde{\nu}_{0}/dm_{\widetilde{\nu}_{0}}$. Then $\log(\zeta_{\widetilde{\nu}_{0}})$ is $L$-Lipschitz. (See Definition \ref{dfn:L-good}.)} 

Let $I$ be a connected component of $\gamma_{\widetilde{\nu}_{0}} \cap (f^{n}_\omega)^{-1}(B(z,\rho))$. {Since $f^n_\omega\gamma_{\widetilde{\nu}_{0}}$ is everywhere tangent to the unstable cone field, by (6) in Section \ref{subsect:setting2}, the length of $f^n_\omega(I)$ is bounded from above by $2\rho/\sin(\C{\theta}{0}/2)$}. Hence, by \eqref{eqn:Lip.est.Dfn.s} and the assumptions on $\rho'$ and $\C{\rho}{0}$, we have $m_{\gamma_{\widetilde{\nu}_0}}(I)\leq 2\rho\cdot\C{C}{3}\lambda_s^{-n}/\sin(\C{\theta}{0}/2)\leq 1$. Fix $x\in I$. By the fact that $\log\zeta_{\widetilde{\nu}_{0}}$ is $L$-Lipschitz, for any $y\in I$, we have $\zeta_{\widetilde{\nu}_{0}}(y)\leq e^{Ld_{\gamma_{\widetilde{\nu}_{0}}}(x,y)}\zeta_{\gamma{\widetilde{\nu}_{0}}}(x)\leq e^{Lm_{\widetilde{\nu}_{0}}(I)}\zeta_{\widetilde{\nu}_{0}}(x)\leq e^L\zeta_{\widetilde{\nu}_{0}}(x)$. Since $\gamma_{\widetilde{\nu}_0}$ is everywhere tangent to the unstable cone field, we have
\[
\widetilde{\nu}_{0}(I) \leq e^L\zeta_{\widetilde{\nu}_{0}}(x)m_{\gamma_{\widetilde{\nu}_0}}(I)\leq \frac{2e^L}{\sin(\C{\theta}{0}/2) }\frac{\zeta_{\widetilde{\nu}_{0}}(x) \rho}{\widehat{J}^{u}_{\omega,n}(z)}.
\]
On the other hand, by \eqref{eqn:Lip.est.Dfn.s}, we have
\[
(f^{n}_{\omega})^{-1}(B(z,\rho)) \subset B\left((f^{n}_{\omega})^{-1}(z),  \C{C'}{3}(\widehat{J}^s_{\omega,n}(z))^{-1} \rho)\right)\subset B\left((f^{n}_{\omega})^{-1}(z), 2 \C{C'}{3}(\widehat{J}^s_{\omega,n}(z))^{-1} \rho\right).
\]
Let $J$ be the connected component of $\gamma_{\widetilde{\nu}_{0}} \cap B\left((f^{n}_{\omega})^{-1}(z), 2 \C{C'}{3}(\widehat{J}^s_{\omega,n}(z))^{-1} \rho\right)$ containing $I$. Observe that there is only one such component. The length of $\gamma_{\widetilde{\nu}_{0}}$ is bounded from below by $2\C{C}{3}\lambda_s^{-n} \rho$, which is greater than $2\C{C'}{3} (\widehat{J}^s_{\omega,n}(z))^{-1} \rho$ (see \eqref{eqn:Lip.est.Dfn.s}). Since $J$ contains $I$, $J$ intersects $(f^{n}_\omega)^{-1}(B(z,\rho))$ and the length of $J$ is bounded from below by  $\C{C'}{3}(\widehat{J}^s_{\omega,n}(z))^{-1}\rho$. {Choose a sub-segment $J_x$ of $J$ containing $x$ such that the length of $J_x$ is $\C{C'}{3}(\widehat{J}^s_{\omega,n}(z))^{-1}\rho$. Notice that $\C{C'}{3} (\widehat{J}^s_{\omega,n}(z))^{-1}\rho\leq \C{C}{3}\lambda_s^{-n}\rho\leq \C{C}{3}\C{\rho}{0}<1$ due to \eqref{eqn:Lip.est.Dfn.s} and the assumptions on $\rho$ and $\C{\rho}{0}$. By the fact that $\log\zeta_{\widetilde{\nu}_0}$ is $L$-Lipschitz, for any $y\in J_x$, we have $\zeta_{\widetilde{\nu}_{0}}(y)\geq e^{-Ld_{\gamma_{\widetilde{\nu}_{0}}}(x,y)}\zeta_{\widetilde{\nu}_{0}}(x)\geq e^{-Lm_{\gamma_{\widetilde{\nu}_{0}}}(J_x)}\zeta_{\widetilde{\nu}_{0}}(x)\geq e^{-L}\zeta_{\widetilde{\nu}_{0}}(x)$. Hence}
\begin{align*}
\widetilde{\nu}_{0} \left(B\left((f^{n}_\omega)^{-1}(z), 2\C{C'}{3} (\widehat{J}^s_{\omega,n}(z))^{-1}\rho\right)\right) \geq \widetilde{\nu}_{0}(J_x)\geq&  e^{-L}\zeta_{\widetilde{\nu}_{0}}(x)m_{\widetilde{\nu}_{0}}(J_x)\\
=& e^{-L}\C{C'}{3}(\widehat{J}^s_{\omega,n}(z))^{-1}\rho\zeta_{\widetilde{\nu}_{0}}(x).
\end{align*}

Hence, 
\[
1 \leq  e^L(\C{C'}{3})^{-1}\widetilde{\nu}_{0}\left(B\left((f^{n}_\omega)^{-1}(z), 2\C{C'}{3}\left(\widehat{J}^s_{\omega,n}(z)\right)^{-1} \rho\right)\right) \widehat{J}^s_{\omega,n}(z)(\rho\zeta_{\widetilde{\nu}_{0}}(x))^{-1}.
\]

Therefore,
\begin{align*}
 &\widetilde{\nu}_{0}((f^{n}_\omega)^{-1}(B(z, \rho)))\\
 \leq & \displaystyle \frac{2e^L}{\sin(\C{\theta}{0}/2) }\frac{ \rho \zeta_{\widetilde{\nu}_{0}}(x)}{\widehat{J}^{u}_{\omega,n}(z)} e^L(\C{C'}{3})^{-1} \frac{(\widehat{J}^s_{\omega,n}(z))}{\rho \zeta_{\widetilde{\nu}_{0}}(x)} \widetilde{\nu}_{0} \left(B\left((f^{n}_\omega)^{-1}(z), 2\C{C'}{3}(\widehat{J}^s_{\omega,n}(z))^{-1}\rho\right)\right)\\
 = & \displaystyle \C{C'}{9} \frac{\widehat{J}^s_{\omega,n}(z)}{ \widehat{J}^{u}_{\omega,n}(z)} \widetilde{\nu}_{0}\left(B\left((f^{n}_\omega)^{-1}(z), 2\C{C'}{3}(\widehat{J}^s_{\omega,n}(z))^{-1}\rho\right)\right),
\end{align*}
where $\CS{C'}{9}=\frac{2e^{2L}}{\C{C'}{3}\sin(\C{\theta}{0}/2)}$. Hence by \eqref{eqn:maininequality1.1} (used in the fourth (in)equality) and (4) in Section \ref{subsect:setting2} (used in the sixth (in)equality), we have
%
\begin{align*}
& \|(f^n_{\omega})_*\nu_{0}\|^2_{ \rho} \\
=&\frac{1}{\rho^{4}} \int_{\mathbb{T}^2} \left(\int_{{\FFC}} \widetilde{\nu}_{0}((f^{n}_\omega)^{-1}(B(z, \rho)))\ d\widehat{\nu}_{0}(\widetilde{\nu}_{0})\right)^2 dm(z)\\ 
\leq&   \int_{\mathbb{T}^2} \frac{1}{\rho^{4}}  (\C{C'}{9})^2 \frac{(\widehat{J}^s_{\omega,n}(z))^{2}}{ (\widehat{J}^{u}_{\omega,n}(z))^2} \left( \int_{{\FFC}} \widetilde{\nu}_{0}\left(B\left((f^{n}_\omega)^{-1}(z), 2\C{C'}{3}(\widehat{J}^s_{\omega,n}(z))^{-1}\rho \right)\right) d\widehat{\nu}_{0}(\widetilde{\nu}_{0})\right)^2  dm (z)\\
=&(\C{C'}{9} )^2\displaystyle  \int_{\mathbb{T}^2} \frac{(\widehat{J}^s_{\omega,n}(z))^{4}}{(2\C{C'}{3}\rho)^{4}}   \frac{16(\C{C'}{3})^4}{ (\widehat{J}^s_{\omega,n}(z)\widehat{J}^{u}_{\omega,n}(z))^2} \left( {\nu}_{0}\left(B\left((f^{n}_\omega)^{-1}(z), 2\C{C'}{3}(\widehat{J}^s_{\omega,n}(z))^{-1}\rho \right)\right)\right)^2  dm (z)\\
\leq&\C{C''}{9}e^{4\varepsilon n}\int_{\mathbb{T}^2} \frac{(\widehat{J}^s_{\omega,n}(z))^{4}}{(2\C{C'}{3}\rho)^{4}}    \left( {\nu}_{0}\left(B\left((f^{n}_\omega)^{-1}(z), 2\C{C'}{3}(\widehat{J}^s_{\omega,n}(z))^{-1}\rho \right)\right)\right)^2  dm (z)\\
=&\C{C''}{9}e^{4\varepsilon n}\int_{\mathbb{T}^2} \frac{(\widehat{J}^s_{\omega,n}(f^n_\omega(p)))^{4}}{(2\C{C'}{3}\rho)^{4}}    \left( {\nu}_{0}\left(B\left((p, 2\C{C'}{3}(\widehat{J}^s_{\omega,n}(f^n_\omega(p)))^{-1}\rho \right)\right) \right)^2  d(f^n_\omega)_*^{-1}m (p)\\
\leq&\C{C''}{9}e^{5\varepsilon n}\int_{\mathbb{T}^2} \frac{(\widehat{J}^s_{\omega,n}(f^n_\omega(p)))^{4}}{(2\C{C'}{3}\rho)^{4}}    \left( {\nu}_{0}\left(B\left((p, 2\C{C'}{3}(\widehat{J}^s_{\omega,n}(f^n_\omega(p)))^{-1}\rho \right)\right) \right)^2  dm (p)=\C{C''}{9}e^{5\varepsilon n}\|\nu_0\|_{\delta}.
\end{align*}
where $\CS{C''}{9}=16\C{C'}{9}(\C{C'}{3})^4$ and $\delta(p) = 2\C{C'}{3}(\widehat{J}^s_{\omega,n}(f^n_\omega(p)))^{-1} \rho$. (See \eqref{generalized r-norm})

Recall that {$\C{C}{9}=2\C{C'}{3}(\C{C''}{0})^{-1}$} and that $\lambda_s= \lambda_{s,-}$. Therefore, we have $\delta(\mathbb{T}^2) \subset [\C{C}{9}\lambda_{s,+}^{-n} \rho, 2\C{C}{3}\lambda_{s}^{-n} \rho]$. (See \eqref{eqn:Lip.est.Dfn.s} and (10) in Section \ref{subsect:other.no.}.) By Lemma \ref{generalized large<small}, 
\[
{ \|(f^n_{\omega})_*\nu_{0}\|^2_{ \rho}} \leq \C{C''}{9}e^{5\varepsilon n}\|\nu_0\|_{\delta}\leq{ \C{C''}{9}e^{5\varepsilon n}\C{C}{2} \left(1 + 2\log(\C{C''}{0})+n \log\left(\frac{\lambda_{s,+}}{\lambda_{s}}\right) \right)}\|\nu_{0}\|_{\C{C}{9}\lambda_{s,+}^{-n} \rho}^2
\]
The lemma then follows from choosing {$\C{n}{4}\geq \C{n}{1}(\varepsilon)$} large enough such that 
\[
{ \C{C''}{9}\C{C}{2} \left(1 + 2\log(\C{C''}{0})+n\log\left(\frac{\lambda_{s,+}}{\lambda_{s,-}}\right) \right)} < e^{\varepsilon n},~\forall n\geq \C{n}{4}.
\]
\end{proof}


\subsection{Transversality}
Let 
\begin{align}\label{eqn:lambda.def}
\C{C}{10}=\max\{2\C{C}{4},\C{L}{0}\}\text{ and }\lambda=\max\left\{\lambda_s^\theta,{\frac{\lambda_{s,+}}{\lambda_{u,-}}}\right\}\in(0,1).
\end{align}
(See Proposition \ref{prop:UH} for the definition of $\C{L}{0}$ and $\theta$.) For each $p\in \mathbb{T}^{2}$, $n\in\mathbb{N}$, and $\delta>0$, we define

\begin{align}\label{eqn:transversality}
\begin{split}
&\mathcal{E}(p, n,\delta)\\
=&\left\{(\omega_1,\omega_2)\left|
\begin{aligned}
&\omega_1,\omega_2\in\Sigma^+\text{ and for any line }F_i\text{ in }\mathcal{C}^u_{(f^{n}_{\omega_i})^{-1}(p)}, i=1,2\\
&\sphericalangle(Df^n_{\omega_1}((f^{n}_{\omega_1})^{-1}(p))F_1, Df^n_{\omega_2} ((f^{n}_{\omega_2})^{-1}(p)) F_2 )\leq 5\C{C}{10}\lambda^n e^{\delta n }.
\end{aligned}
\right.\right\}
\end{split}
\end{align}

Roughly speaking, if $(\omega_1,\omega_2)\not\in\cE(p,n,\delta)$, then the pair of cones $Df^n_{\omega_1}((f^{n}_{\omega_1})^{-1}(p))\cC^u_{(f^{n}_{\omega_1})^{-1}(p)}$ and $Df^n_{\omega_2}((f^{n}_{\omega_2})^{-1}(p))\cC^u_{(f^{n}_{\omega_2})^{-1}(p)}$ are ``transverse'' to each other. Here, two cones are ``transverse'' to each other if there is a large angle between them. 
\begin{lem}\label{lem:transversality.nearby}
 For any $\rho'\in(0,1)$,  $p\in\TT^2$,  $\omega\in\Sigma^+$,  $n>0$,  $\delta>0$,  and $q\in B(p,\lambda_s^n\rho')$,  and for any lines $F$ in $\cC^u_{(f^{n}_{\omega})^{-1}(q)}$ and $F'$ in $\cC^u_{(f^n_\omega)^{-1}(p)}$, we have 
$$\sphericalangle(Df^n_{\omega}((f^{n}_{\omega})^{-1}(q))F,  Df^n_{\omega} ((f^{n}_{\omega})^{-1}(p))F' ) \leq  2\C{C}{10}\lambda^n e^{\delta n }.$$

\end{lem}
\begin{proof}

Fix $\omega = (f_0, f_1, \cdots) \in \Sigma^+$,  and $\omega^- = (\cdots, f_{-2}', f_{-1}')\in \mathcal{U}^{-\mathbb{N}}$. Consider $\omega' = (\dots,f_{-2}', f_{-1}', f_0, f_1, \cdots, f_{n-1},\dots) \in \mathcal{U}^{\Zbb}$. The negative coordinate of the word $\omega'$ is $\omega^-$ and the non-negative coordinate of the word $\omega'$ is $\omega$. Note that unstable distributions and unstable manifolds only depend on the past and so $E^{u}_{\omega',p}=E^{u}_{\omega^{-},p}$.
By Proposition \ref{prop:UH}, \eqref{eqn:cone.contraction.u} and (6) in Section \ref{subsect:setting2}, we have
\begin{align*}
&\sphericalangle(Df^n_{\omega}((f^{n}_{\omega})^{-1}(q))F,  Df^n_{\omega} ((f^{n}_{\omega})^{-1}(p))F' ) \\
\leq &\sphericalangle(Df^n_{\omega}((f^{n}_{\omega})^{-1}(q))F,  Df^n_{\omega}((f^{n}_{\omega})^{-1}(q))E^u_{\omega^{-}, (f^n_\omega)^{-1}(q)} ) \\
&+\sphericalangle(Df^n_{\omega}((f^{n}_{\omega})^{-1}(q)) E^u_{\omega^-,(f^{n}_{\omega})^{-1}(q)},  Df^n_{\omega} ((f^{n}_{\omega})^{-1}(p)) E^u_{\omega^-,(f^{n}_{\omega})^{-1}(p)} ) \\
&+\sphericalangle(Df^n_{\omega}((f^{n}_{\omega})^{-1}(p)) E^u_{\omega^-,(f^{n}_{\omega})^{-1}(p)},  Df^n_{\omega} ((f^{n}_{\omega})^{-1}(p)) F' ) \\
=&\sphericalangle(Df^n_{\omega}((f^{n}_{\omega})^{-1}(q))F,  Df^n_{\omega}((f^{n}_{\omega})^{-1}(q))E^u_{\omega^-, (f^n_\omega)^{-1}(q)} )+\sphericalangle(E^u_{\sigma^{n}(\omega'),q},E^u_{\sigma^{n}(\omega'),p})\\
&+\sphericalangle(Df^n_{\omega}((f^{n}_{\omega})^{-1}(p)) E^u_{\omega^-,(f^{n}_{\omega})^{-1}(p)},  Df^n_{\omega} ((f^{n}_{\omega})^{-1}(p)) F' )\\
\leq &2\C{C}{4}\left(\frac{\lambda_{s,+}}{\lambda_{u,-}}\right)^n+\C{L}{0}d(p,q)^\theta\leq 2\C{C}{4}\left(\frac{\lambda_{s,+}}{\lambda_{u,-}}\right)^n+\C{L}{0}\lambda_s^{n\theta}.
\end{align*}
The lemma then follows from \eqref{eqn:lambda.def}. 
\end{proof}
\begin{rmk}
Let $(\omega_1,\omega_2)\not \in\cE(p,n,\delta)$. Fix an arbitrary $\rho'\in(0,1/2]$ and an arbitrary $p\in\TT^2$. Suppose $q_1,q_2,q_1',q_2'\in B(p,\lambda^n_sp)$ are points such that for any $i=1,2$, there exists a $C^1$-curve in $B(p,\lambda^n_s\rho')$ connecting $q_i$ and $q_i'$ which is everywhere tangent to the cone field $Df^n_{\omega_i}(\cC^u)$. Then the angle between $\overline{q_1q_1'}$ and $\overline{q_2q_2'}$ is at least $\C{C}{10}\lambda^n e^{\delta n}$, where $\overline{q_iq_i'}$ is the shortest geodesic segment connecting $q_i$ and $q_i'$, $i=1,2$. 
\end{rmk}
\begin{lem}\label{lem:nontrans.est}
 For any $\delta>0$, there exists a constant $\CS{n}{5}=\C{n}{5}(\delta)>0$ such that for any $\rho'\in(0,\C{\rho}{0})$,  $p\in\TT^2$,  $\omega\in\Sigma^+$,  $n\geq\C{n}{5}$ and  $r\geq \max\{\lambda e^{2\delta},\eta\}$, we have
$$\mu^{\mathbb{N}}(\{\omega'\in\Sigma^+|(\omega,\omega')\in\cE(p,n,\delta)\})\leq \C{C}{5}r^{\alpha n}.$$
See Proposition \ref{prop.holderness} for $\eta=\eta(\beta, \C{\theta}{0},\C{\theta}{\Delta})$ and $\alpha=\alpha(\beta, \C{\theta}{0},\C{\theta}{\Delta})$.
\end{lem}
\begin{proof}
Fix a continuous line field $p\to F_p\in\PP\cC^u_p$, where $\PP\mathcal{C}^u_p$ is the projection of $\mathcal{C}^u_p$ in the projective space $\mathbb{P}T_p\mathbb{T}^2$. Let $\widehat{\nu}_p\in\mathrm{Prob}(\PP T_p\TT^2)$ be the Dirac mass at $F_p$. Then for any $p\in\TT^2$, by Proposition \ref{prop.holderness}, we have 
$$(\mu^{*n}*\widehat{\nu})_p\left(B_{r^n}\left(Df^n_{\omega}((f^{n}_\omega)^{-1}(p))F_{(f^{n}_\omega)^{-1}(p)}\right)\right)\leq \C{C}{5}r^{n\alpha},~\forall n>0.$$
Choose $\C{n}{5}(\delta)>0$ such that $e^{\C{n}{5}\delta}>5\C{C}{10}$. Then by \eqref{eqn:transversality}, for any $n\geq \C{n}{5}$ and for any $\omega'\in\Sigma^+$ such that $(\omega,\omega')\in\cE(p,n,\delta)$, we have 
$$\sphericalangle(Df^n_{\omega'}((f^{n}_{\omega'})^{-1}(p))F_{(f^{n}_{\omega'})^{-1}(p)},Df^n_{\omega}((f^{n}_\omega)^{-1}(p))F_{(f^{n}_\omega)^{-1}(p)})\leq r^n.$$
Therefore
\begin{align*}
\mu^{\mathbb{N}}(\{\omega'\in\Sigma^+|(\omega,\omega')\in\cE(p,n,\delta)\})\leq& (\mu^{*n}*\widehat{\nu})_p\left(B_{r^n}\left(Df^n_{\omega}((f^{n}_\omega)^{-1}(p))F_{(f^{n}_\omega)^{-1}(p)}\right)\right)\\
\leq &\C{C}{5}r^{n\alpha},~\forall n\geq \C{n}{5}.\qedhere
\end{align*}
\end{proof}

\subsection{A Lasota-Yorke type of estimate}


In the setting in \Cref{subsect:setting} and \Cref{subsect:setting2}, by Theorem \ref{thm.uniquesrb}, there is a unique $\mu$-stationary SRB measure.  From now on, let $\nu$ be the unique $\mu$-stationary SRB measure.

Fix $\xi$ a $u$-subordinated measurable partition. Hence, 
\[
\displaystyle \nu = \int_{\Sigma} \int_{\mathbb{T}^2} \nu^u_{(\omega, x)} d\nu_{\omega}(x) d\mu^{\mathbb{Z}}(\omega),
\]
where family $\{\nu_\omega \}_{\omega}$ is the family of sample measures, and $\nu^u_{(\omega,x)}$ is the conditional measure of $\nu_{\omega}$ on $\xi(\omega,x)$.  Since $\nu$ is an SRB measure, we have that for $\mu^{\mathbb{Z}}$-almost every $\omega$ and for $\nu_{\omega}$-almost every $x$, the measure $\nu^u_{(\omega,x)}$ is a probability measure absolute continuous with respect to the arc-length measure $m_{\xi(\omega, x)}$ in $\xi(\omega,x)$. Moreover, for such $(\omega,x)$, there exists a constant $L>0$ independent of $(\omega,x)$ such that the density function $\rho^{u}_{(\omega,x)}:=d\nu^u_{(\omega,x)}/dm_{\xi(\omega, x)}$ is positive and that $\log\rho^{u}_{(\omega,x)}$ is $L$-Lipschitz (see \Cref{thm.log.lip.density}). From now on, we assume, without loss of generality, that $L\geq \C{L}{1}(\cU)$ (recall that $L_1(\mathcal{U})$ is given by \Cref{lem.goodmeasurespreserved}). For the remaining parts of this paper, when we introduce more constants, we do not track their dependence on $L$.

\subsubsection*{$(\mathcal{U}, L)$-admissible measure $\nu_{r}$} For each $r>0$, consider $\mathcal{G}_r:=\{(\omega,x)\in \Sigma \times \mathbb{T}^2: |\xi(\omega,x)| \geq r\}$. This is the set of points $(\omega,x)$ such that $\xi(\omega,x)$ has length at least $r$.  Define
\[
\nu_r = \displaystyle \iint_{\mathcal{G}_r} \nu^u_{(\omega,x)} d\nu_{\omega}(x) d\mu^{\mathbb{Z}}(\omega).
\]
This is the part of the measure $\nu$ supported on unstable curves with length bounded from below by $r$.  Observe that  $\nu_r$ is a (${\mathcal{U}}$,$L$)-admissible measure.

Let $\widehat{\nu}$ be the lift of the measure $\mu^{\mathbb{N}} \otimes \nu$, in $\Sigma^+\times \mathbb{T}^2$, to $\Sigma \times \mathbb{T}^2$.  
\begin{lem}\label{lemma.srbconvergence}
For $r$ sufficiently small,  the following properties hold:
\begin{enumerate}
\item[(a)] $\displaystyle \widehat{\nu}(\mathcal{G}_r) =  \iint_{\mathcal{G}_r} d\nu_{\omega}(x) \mu^{\mathbb{Z}}(\omega) >0.$
\item[(b)] $\displaystyle \lim_{n\to +\infty} \frac{1}{n} \sum_{j=0}^{n-1} \frac{1}{\widehat{\nu}_r(\mathcal{G}_r)}\mu^j_* \nu_r = \nu.$
\end{enumerate}

\end{lem}
 
 Item (a) is a direct consequence of the fact that $\lim_{r \to 0} \hat{\nu}(\mathcal{G}_r) = 1$. Item (b) follows from the ergodicity of $\nu$.
 
 \begin{Remark}\label{remark.lengthpartition}
 Take  $(\omega,x) \in \mathcal{G}_r$.  {By (10) in Section \ref{subsect:other.no.}}, for any $\omega' \in \Sigma^+$, the measure $(f^n_{\omega'})_* \nu^u_{(\omega,x)}$ is a probability measure supported on $f^n_{\omega'}(\xi(\omega,x))$ which has length bounded from below by $(\C{C''}{0})^{-1}\lambda_{u,-}^n r$.  Hence, the measure $\mu^{*n}*\nu_r$ admits a disintegration by measures supported on unstable curves with length bounded from below by $(\C{C''}{0})^{-1}\lambda_{u,-}^n r$.   
\end{Remark}

\subsubsection*{Constants $\rho_1,\rho_2(n)>0$} Fix $\CS{\rho}{1}\in(0,\C{\rho}{0}/3)$ small such that $\widehat{\nu}(\mathcal{G}_{\C{\rho}{1}})>0$,  and for each $n\in \mathbb{N}$, let $\CS{\rho}{2} = \C{\rho}{2}(n)  =(\C{C}{3}\lambda_{u,+}^n)^{-1} \lambda_s^{2n} \C{\rho}{1}$. {Let $\Gamma_n$ be the collection of points $(x_1,x_2)$ in $\mathbb{T}^2$ such that $x_1,x_2\in([2/\C{\rho}{2}]+1)^{-1}\cdot\ZZ$.} In particular, 
\begin{align}\label{eqn:lattice}
{\bigcup_{x\in \Gamma_n} B(x,\C{\rho}{2}) = \mathbb{T}^2\text{~and~}\sup_{z\in\TT^n}|\{x\in\Gamma_n|z\in B(x,3\C{\rho}{2})\}|\leq 10^2.}
\end{align}
The following key Lasota--Yorke type inequality closely follows \cite[Lemma 6.5]{Tsujii-bigpaper}. 
\begin{lem}\label{lem.lasotayork}
Let $\nu'$ be an arbitrary $(\cU,L)$-admissible measure supported on curves of length bounded from below by $\C{\rho}{1}$. Then there exist constants $\CS{C}{11}(n,\C{\rho}{1})>0$, $\CS{n}{6}=\C{n}{6}(\varepsilon)>0$ and $\widehat{\lambda} \in (0,1)$ independent of the choice of $\nu'$, such that for any $n\geq \C{n}{6}$ and for any $\rho$ such that
\begin{align}\label{eqn:rho.cond}
0<\rho<\min\left\{\frac{\C{C}{10}\lambda^n\C{\rho}{1}}{10},\frac{\C{\rho}{2}(n)}{\C{C}{3}\lambda_{u,+}^n\cdot \C{C}{3}\lambda_s^{-n}}\right\},
\end{align}
we have
$$\|\mu^{*n}*\nu'\|^2_{\rho} \leq \widehat{\lambda}^n \|\nu'\|^2_{  \rho} + \C{C}{11}(n,\C{\rho}{1})(\nu'(\TT^2))^2.$$

\end{lem}

\begin{proof}
For each $n\in \mathbb{N}$, let $\rho,$ $\C{\rho}{2}$ and $\Gamma_n$ be as above. By \eqref{eqn:lattice}, we have 
\begin{align}\label{eqn:chopping.1}
\|\mu^{*n}* \nu' \|^2_{\rho} \leq 10^2 \displaystyle \sum_{p\in \Gamma_n} \|\mu^{*n}* \nu'|_{B(p,\C{\rho}{2})}\|^2_{\rho} 
\end{align}

Fix $p\in\Gamma_n$. Let us estimate $\|\mu^n* \nu'|_{B(p,\C{\rho}{2})}\|^2_{\rho}$. For any $\omega_1,\omega_2\in\Sigma^+$, we say that $\omega_1$ is transverse to $\omega_2$, if $(\omega_1, \omega_2) \notin \mathcal{E}(p,n,\varepsilon)$. We will write $\omega_1 \pitchfork \omega_2$ whenever $\omega_1$ is transverse to $\omega_2$. Otherwise, we write $\omega_1\parallel\omega_2$. 

We want to estimate
\begin{align}\label{eqn:trans.decomp}
\|\mu^{*n}* \nu'|_{B(p,\rho)}\|^2_{\rho} 
=&  \left\langle \int (f^n_{\omega_1})_*\nu'|_{B(p,\C{\rho}{2})}d{\mu^{\mathbb{N}}}(\omega_1), \int(f^n_{\omega_2})_*\nu'|_{B(p,\C{\rho}{2})}d{\mu^{\mathbb{N}}}(\omega_2) \right\rangle_\rho\nonumber\\
=& \iint_{\{\omega_1 \pitchfork \omega_2\}} \left\langle (f^n_{\omega_1})_* \nu'|_{B(p,\C{\rho}{2})}, (f^n_{\omega_2})_*\nu'|_{B(p,\C{\rho}{2})} \right\rangle_{\rho} d{\mu^{\mathbb{N}}}(\omega_1) d{\mu^{\mathbb{N}}}(\omega_2)\nonumber\\
& +\iint_{\{\omega_1  \parallel \omega_2\}} \left\langle (f^n_{\omega_1})_* \nu'|_{B(p,\C{\rho}{2})}, (f^n_{\omega_2})_*\nu'|_{B(p,\C{\rho}{2})} \right\rangle_{\rho} d{\mu^{\mathbb{N}}}(\omega_1) d{\mu^{\mathbb{N}}}(\omega_2)\nonumber\\
=& I + II.
\end{align}
For each $\omega$ and for any $r>0$, we write $D^n_\omega(p,r) = (f^{n}_\omega)^{-1}(B(p,r))$. 

Let us first estimate $II$. Observe that
\begin{align}\label{eqn:II-main-1}
II\leq &\iint_{\{\omega_1\parallel \omega_2\}} \frac{1}{2} \left(\|(f^n_{\omega_1})_* \nu'|_{B(p,\C{\rho}{2})}\|_\rho^2 + \|(f^n_{\omega_2})_*\nu'|_{B(p,\C{\rho}{2})}\|^2_\rho\right)d{\mu^{\mathbb{N}}}(\omega_1) d{\mu^{\mathbb{N}}}(\omega_2) \nonumber\\
 =&\int_{\Sigma^+}\mu^{\mathbb{N}}(\{\omega_2:\omega_1\parallel\omega_2\})\cdot\|(f^n_{\omega_1})_* \nu'|_{B(p,\C{\rho}{2})}\|_\rho^2d\mu^{\mathbb{N}}(\omega_1).
\end{align}
We would like to apply Lemma \ref{lemma.maininequality1} to $ \|(f^n_{\omega_1})_*\nu'|_{B(p,\C{\rho}{2})}\|^2_\rho$. However, the measure $\nu'|_{D^n_{\omega_1}(p,\C{\rho}{2})}$, which is an admissible measure,  might not be supported on admissible curves with length bounded from below by $2
\C{C}{3}\lambda_s^{-n} \rho$.

Let $\gamma$ be  an admissible curve with length greater than $2 \C{C}{3}\lambda_s^{-n}\rho$ intersecting \linebreak $D^n_{\omega_1}(p, 3\C{\rho}{2})$, and let $\gamma'$ be a connected component of $\gamma \cap D^n_{\omega_1}(p,3\C{\rho}{2})$ with length smaller than $2\C{C}{3}\lambda_s^{-n} \rho$. Hence, $\gamma'$ must intersect the boundary of $D^n_{\omega_1}(p,3\C{\rho}{2})$. By \eqref{eqn:Lip.est.Dfn.u} and the assumptions on $\rho$, $\gamma'$ does not intersect $D^n_{\omega_1}(p,\C{\rho}{2})$. 

Consider $\nu'|_{D^n_{\omega_1}(p,3\C{\rho}{2})}$ and let $\widetilde{\nu}$ be the measure obtained by discarding the part of the measure supported on small admissible curves (smaller than $2\C{C}{3}\lambda_s^{-n}\rho$) from $\nu'|_{D^n_{\omega_1}(p,3\C{\rho}{2})}$.  It follows from the above discussion that $\nu'|_{D^n_{\omega_1}(p,\C{\rho}{2})} \leq \widetilde{\nu} \leq \nu'|_{D^n_{\omega_1}(p,3\C{\rho}{2})}$. By Lemma \ref{lemma.maininequality1}, for any $n\geq \C{n}{4}(\varepsilon)$, we have
\begin{align}\label{eqn:II-est-1}
\begin{split}
\|(f^n_{\omega_1})_*\nu'|_{B(p,\C{\rho}{2})}\|^2_\rho \leq \|(f^n_{\omega_1})_* \widetilde{\nu}\|^2_\rho \leq &e^{6\varepsilon n} \|\widetilde{\nu}\|^2_{\C{C}{9}\lambda_{s,+}^{-n} \rho }\\
 \leq &e^{6\varepsilon n} \|\nu'|_{D^n_{\omega_1}(p,3\C{\rho}{2})}\|^2_{\C{C}{9}\lambda_{s,+}^{-n} \rho }.
\end{split}
\end{align}
Observe that the same estimate works for $\omega_2$.

Take  $\widehat{\eta}= \max\{\lambda e^{2\varepsilon}, \eta\}$. By Lemma \ref{lem:nontrans.est}, for any $n\geq \C{n}{5}(\varepsilon)$, we have  
\begin{align}\label{eqn:II-est-2}
\mu^{\mathbb{N}}(\{\omega_2:\omega_1 \parallel \omega_2\}) \leq \C{C}{5} \widehat{\eta}^{\alpha n},~\forall \omega_1\in\Sigma^+.
\end{align}
Since for each $\omega$, $(f^n_\omega)^{-1}$ is a diffeomorphism, by \eqref{eqn:lattice}, $\{D^n_\omega(p,3\C{\rho}{2})\}_{p \in \Gamma_n}$ form a finite cover of $\mathbb{T}^2$ whose maximum number of overlaps is bounded from above by $10^2$. Thus, by \eqref{eqn:II-main-1} (used in the first inequality), \eqref{eqn:II-est-1} (used in the second inequality), \eqref{eqn:II-est-2} (used in the first inequality) and the above (used in the third inequality), we have
\begin{align*}
\sum_{p\in\Gamma_n}II\leq& \sum_{p\in\Gamma_n}\int_{\Sigma^+}\mu^{\mathbb{N}}(\{\omega_2:\omega_1\parallel\omega_2\})\cdot\|(f^n_{\omega_1})_* \nu'|_{B(p,\C{\rho}{2})}\|_\rho^2d\mu^{\mathbb{N}}(\omega_1)\\
\leq& \C{C}{5} \widehat{\eta}^{\alpha n}e^{6\varepsilon n}\int_{\Sigma^+} \sum_{p\in\Gamma_n}\|\nu'|_{D^n_{\omega_1}(p,3\C{\rho}{2})}\|^2_{\C{C}{9}\lambda_{s,+}^{-n} \rho}d\mu^{\mathbb{N}}(\omega_1)\\
\leq &10^4\C{C}{5} \widehat{\eta}^{\alpha n}e^{6\varepsilon n}\|\nu'\|^2_{\C{C}{9}\lambda_{s,+}^{-n} \rho }.
\end{align*}
By our choice of $\varepsilon$ in Section \ref{subsect:setting2}, we have that $\widehat{\lambda}:=\widehat{\eta}^\alpha e^{7\varepsilon}<1$. Let $\CS{n'}{6}=\C{n'}{6}(\varepsilon)>\max\{\C{n}{4}(\varepsilon),\C{n}{5}(\varepsilon)\}$ such that $\max\{\C{C}{9}\lambda_{s,+}^{-\C{n'}{6}}  ,10^4\C{C}{1}\C{C}{5}e^{-\varepsilon \C{n'}{6}}\}<1$. Then by Lemma \ref{lem.changeofscale}, for any $n\geq \C{n'}{6}(\varepsilon)$, we have
\begin{align}\label{eqn:II}
\sum_{p\in\Gamma_n}II\leq\C{C}{1}\cdot10^4\C{C}{5}e^{-\varepsilon n}\widehat{\lambda}^n\|\nu'\|^2_\rho\leq \widehat{\lambda}^n\|\nu'\|^2_\rho
\end{align}

We then estimate $I$.  We would like to show that there exists a constant $\CS{C'}{11}(n,\C{\rho}{1})>0$ and $\C{n''}{6}>0$ such that for any $n\geq \C{n''}{6}$, we have
\begin{align}\label{eqn:I.integrand.est}
\left\langle (f^n_{\omega_1})_* \nu'|_{B(p,\C{\rho}{2})}, (f^n_{\omega_2})_*\nu'|_{B(p,\C{\rho}{2})} \right\rangle_{\rho}\leq\C{C'}{11}(n,\C{\rho}{1})\nu'(D^n_{\omega_1}(p,3\C{\rho}{2}))\nu'(D^n_{\omega_2}(p,3\C{\rho}{2})).
\end{align}
Assume that \eqref{eqn:I.integrand.est} is true, then for any $n\geq\C{n''}{6}$, we have 
\begin{align}\label{eqn:I}
\sum_{p\in\Gamma_n}I\leq &\sum_{p\in\Gamma_n}\iint_{\{\omega_1 \pitchfork \omega_2\}}\C{C'}{11}(n,\C{\rho}{1})\nu'(D^n_{\omega_1}(p,3\C{\rho}{2}))\nu'(D^n_{\omega_2}(p,3\C{\rho}{2})) d{\mu^{\mathbb{N}}}(\omega_1) d{\mu^{\mathbb{N}}}(\omega_2)\nonumber\\
\leq &\C{C'}{11}(n,\C{\rho}{1})|\Gamma_n|\cdot(\nu'(\TT^2))^2.
\end{align}
Notice that $|\Gamma_n|$ only depends on $n$ and $\C{\rho}{1}$. Choose $\C{C}{11}(n,\C{\rho}{1}):=\C{C'}{11}(n,\C{\rho}{1})|\Gamma_n|$ and $\C{n}{6}(\varepsilon):=\max\{\C{n'}{6}(\varepsilon),\C{n''}{6}\}$. Then the lemma follows from \eqref{eqn:II} and \eqref{eqn:I}.

It remains for us to verify \eqref{eqn:I.integrand.est}. Observe that $(f^n_{\omega_i})_*\nu'|_{B(p,3\C{\rho}{2})} = \nu'|_{D^n_{\omega_i}(p,3\C{\rho}{2})}  \circ (f^{n}_{\omega_i})^{-1}$. Since both sides of \eqref{eqn:I.integrand.est} are bilinear in $ \nu'|_{D^n_{\omega_i}(p,3\C{\rho}{2})}$, by the assumption on $\nu'$, we can assume without loss of generality that there exists an admissible curve $\gamma_i'$ with length at least $\C{\rho}{1}$ satisfying the following:
\begin{itemize}
\item $ \nu'|_{D^n_{\omega_i}(p,3\C{\rho}{2})}$ is supported on a connected component $\gamma_i$ of $\gamma_i'\cap D^n_{\omega_i}(p,3\C{\rho}{2})$. 
\item $ \nu'|_{D^n_{\omega_i}(p,3\C{\rho}{2})}$ is absolutely continuous with respect to the arclength measure $m_{\gamma_i}$ of $\gamma_i$ with an $L$-log-Lipschitz density.
\end{itemize} 
If $\gamma_i$ has length less than $2(\C{C}{3}\lambda_{u,+}^n)^{-1}\C{\rho}{2}$,  then by the fact that $\C{\rho}{1}\geq (\C{C}{3}\lambda_{u,+}^n)^{-1}\C{\rho}{2}$, $\gamma_i$ must intersect the boundary of $D^n_{\omega_i}(p,3\C{\rho}{2})$. By \eqref{eqn:Lip.est.Dfn.u}, $\gamma_i$ does not intersect $D^n_{\omega_i}(p,\C{\rho}{2})$ and hence \eqref{eqn:I.integrand.est} holds trivially. Therefore, we  assume, without loss of generality,  that
\begin{align}\label{eqn:I.length.lowerbd}
m_{\gamma_i}(\gamma_i)=m_{\gamma_i}(D_{\omega_i}^n(p,3\C{\rho}{2}))\geq 2(\C{C}{3}\lambda_{u,+}^n)^{-1}\C{\rho}{2}=2(\C{C}{3}\lambda_{u,+}^n)^{-2}\lambda_s^{2n}\C{\rho}{1}.
\end{align}

Let $\nu_i' = \nu'|_{D^n_{\omega_i}(p,3\C{\rho}{2})}$. By \eqref{eqn:Lip.est.Dfn.s} and the definition of $\C{\rho}{2}$, we have
$$D^n_{\omega_i}(p,3\C{\rho}{2})\subset B((f^{n}_{\omega_i})^{-1}(p),3\lambda^{-n}_{u,+}\lambda^n_s\C{\rho}{1})\subset B((f^{n}_{\omega_i})^{-1}(p),\lambda^{-n}_{u,+}\lambda^n_s\C{\rho}{0}).$$
Hence, by (6) in Section \ref{subsect:setting2} and the fact that $\gamma_i$ is everywhere tangent to the unstable cone field, for any $z\in\TT^2$, we have \begin{align}\label{eqn:I.length.upperbd}
m_{\gamma_i}(\gamma_i)\leq \frac{2\lambda^{-n}_{u,+}\lambda^n_s\C{\rho}{0}}{\sin(\C{\theta}{0}/2)}=(\C{C}{3}\lambda_{u,+}^n)^{-1}\cdot \frac{2\lambda_s^n\C{C}{3}\C{\rho}{0}}{\sin(\C{\theta}{0}/2)}\leq \frac{1}{2}(\C{C}{3}\lambda_{u,+}^n)^{-1}\lambda_s^n< \frac{1}{2}.
\end{align}
As a corollary of \eqref{eqn:I.length.lowerbd}, \eqref{eqn:I.length.upperbd} and the $L$-log-Lipschitz property of $d\nu_i'/dm_{\gamma_i}$, we have
$$\frac{\nu_i'(D^n_{\omega_i}(z,\rho))}{\nu_i'(D^n_{\omega_i}(p,3\C{\rho}{2}))}\leq e^{L}\frac{m_{\gamma_i}(D^n_{\omega_i}(z,\rho))}{m_{\gamma_i}(D^n_{\omega_i}(p,3\C{\rho}{2}))}\leq\frac{e^{L}(\C{C}{3}\lambda^n_{u,+})^2}{2\lambda_s^{2n}\C{\rho}{1}}m_{\gamma_i}(D^n_{\omega_i}(z,\rho)).$$
Therefore, we obtain 
\begin{align}\label{eqn:I-1}
&\frac{\left\langle (f^n_{\omega_1})_* \nu'|_{B(p,\C{\rho}{2})}, (f^n_{\omega_2})_*\nu'|_{B(p,\C{\rho}{2})} \right\rangle_{\rho}}{\nu'(D^n_{\omega_1}(p,3\C{\rho}{2}))\nu'(D^n_{\omega_2}(p,3\C{\rho}{2}))}\nonumber\\
=&\left(\frac{1}{\rho^4}\int_{\TT^2}\nu_1'(D^n_{\omega_1}(z,\rho))\nu_2'(D^n_{\omega_2}(z,\rho))dm(z)\right)\cdot\frac{1}{\nu'(D^n_{\omega_1}(p,3\C{\rho}{2}))\nu'(D^n_{\omega_2}(p,3\C{\rho}{2}))}\nonumber\\
\leq&\left(\frac{e^{L}(\C{C}{3}\lambda^n_{u,+})^2}{2\lambda_s^{2n}\C{\rho}{1}}\right)^2\cdot\frac{1}{\rho^4}\int_{\TT^2}m_{\gamma_1}(D^n_{\omega_1}(z,\rho))m_{\gamma_2}(D^n_{\omega_2}(z,\rho))dm(z)\nonumber\\
=&\C{C''}{11}(n,\C{\rho}{1})\int_{\mathbb{T}^2} \frac{1}{\rho^{4}} \int_{\gamma_1 \times \gamma_2} \mathbbm{1}_{\rho}(f^n_{\omega_1}(x), z) \mathbbm{1}_\rho(f^n_{\omega_2}(y), z) dm_{\gamma_1}(x) dm_{\gamma_2}(y) dm(z)
\end{align}
where $\CS{C''}{11}(n,\C{\rho}{1}):=\left(\frac{e^{2}(\C{C}{3}\lambda^n_{u,+})^2}{2\lambda_s^{2n}\C{\rho}{1}}\right)^2$ and
\[
\mathbbm{1}_\rho(x,y) = 
\begin{cases}
1, \textrm{  if }d(x,y) \leq \rho,\\
0, \textrm{ otherwise.}
\end{cases}
\]
However,
\[
\mathbbm{1}_{\rho}(f^n_{\omega_1}(x), z) \mathbbm{1}_\rho(f^n_{\omega_2}(y), z)  \leq \mathbbm{1}_{2\rho}(f^n_{\omega_1}(x) , f^n_{\omega_2}(y)) \mathbbm{1}_\rho(f^n_{\omega_2}(y), z).
\]
Hence by (7) in Section \ref{subsect:other.no.} and the above, we have
\begin{align}\label{eqn:I-2}
 &\int_{\mathbb{T}^2} \frac{1}{\rho^{4}} \int_{\gamma_1 \times \gamma_2} \mathbbm{1}_{\rho}(f^n_{\omega_1}(x), z) \mathbbm{1}_\rho(f^n_{\omega_2}(y), z) dm_{\gamma_1}(x) dm_{\gamma_2}(y) dm(z) \nonumber\\
\leq& \int_{\gamma_1 \times \gamma_2} \frac{1}{\rho^{4} }  \mathbbm{1}_{2\rho}(f^n_{\omega_1}(x) , f^n_{\omega_2}(y))\left( \int_{\mathbb{T}^2} \mathbbm{1}_\rho(f^n_{\omega_2}(y), z)dm(z) \right) dm_{\gamma_1}(x) dm_{\gamma_2}(y)  \nonumber\\
\leq&\frac{\C{C}{0}\pi}{\rho^2}\int_{\gamma_1 \times \gamma_2} \mathbbm{1}_{2\rho}(f^n_{\omega_1}(x) , f^n_{\omega_2}(y))dm_{\gamma_1}(x) dm_{\gamma_2}(y).
\end{align}

For any $i=1,2$, we write $\gamma_i^n = f_{\omega_i}^n(\gamma_i)$ and let $m_{\gamma_i^n}$ be the arc length measure on $\gamma_i^n$. We first notice that for each $y_n \in \gamma_2^n$, by (6) in Section \ref{subsect:setting2}, we have 
\begin{align*}
m_{\gamma_1^n}(\{x_n\in \gamma_1^n: d(x_n,y_n)< 2\rho\}) < \frac{4\rho}{\sin(\C{\theta}{0}/2)}.
\end{align*}
Applying $(f^n_{\omega_1})^{-1}$ to $\gamma_1^n$, it follows from (10) in Section \ref{subsect:other.no.} that 
\begin{align}\label{eqn:I-2-1}
m_{\gamma_1}((f^n_{\omega_1})^{-1}(\{x_n\in \gamma_1^n: d(x_n,y_n)< 2\rho\})) < \frac{4\C{C''}{0}\lambda_{u,-}^{-n}\rho}{\sin(\C{\theta}{0}/2)}.
\end{align}
Assume that there exist $x_n\in\gamma_1^n$ and $y_n\in\gamma_2^n$ such that $d(x_n,y_n)<2\rho$. Observe the following: 
\begin{itemize}
\item $m_{\gamma_i^n}(\gamma_i^n)\leq \lambda_s^n/2$ for any $i=1,2$. (This is due to \eqref{eqn:I.length.upperbd} and \eqref{eqn:Lip.est.Dfn.u}.)
\item Since $\omega_1\pitchfork\omega_2$, for any point $y\in \gamma_2^n$ and any point $x\in \gamma_1^n$ such that $d(x,y) < 2\rho$, the angle between the $\overline{xx_n}$ and $\overline{yy_n}$ is at least $\C{C}{10}\lambda^n e^{\varepsilon n}$, where $\overline{xx_n}$ ($\overline{yy_n}$ resp.) is the geodesic segment connecting $x$ ($y$ resp.) and $x_n$ ($y_n$ resp.) in $\TT^2$. (This follows from the previous bullet point and the remark after Lemma \ref{lem:transversality.nearby}).
\end{itemize}
Choose $\CS{n''}{6}$ such that $\lambda_s^{\C{n''}{6}}/2<\C{C}{3}\C{\rho}{0}$ and that $2\sin(t)>t$ for any $0\leq t\leq  \C{C}{10}\lambda^{\C{n''}{6}}$. Then for any $n\geq \C{n''}{6}$, One can then easily verify that $\gamma_2^n\subset B(y_n,\C{C}{3}\C{\rho}{0})$ and that
$$\{y\in\gamma_2^n: d(y,\gamma_1^n)<2\rho\}\subset B(y_n,(\sin(\C{C}{10}\lambda^n))^{-1}\cdot 4\rho)\subset B(y_n,(\C{C}{10}\lambda^n )^{-1}\cdot 8\rho).$$
By (6) in Section \ref{subsect:setting2} and the assumptions on $\rho$ and $\C{\rho}{1}$, we have 
$$m_{\gamma_2^n}(\{y\in\gamma_2^n: d(y,\gamma_1^n)<2\rho\})\leq \frac{16\rho}{\C{C}{10}\lambda^n \sin(\C{\theta}{0}/2)},~\forall n\geq\C{n''}{6}.$$
Hence by (10) in Section \ref{subsect:other.no.}, we have 
\begin{align}\label{eqn:I-2-2}
m_{\gamma_2}((f^n_{\omega_2})^{-1}(\{y\in\gamma_2^n: d(y,\gamma_1^n)<2\rho\}))\leq \frac{16\C{C''}{0}\lambda_{u,-}^{-n}\rho}{\C{C}{10}\lambda^n \sin(\C{\theta}{0}/2)},~\forall n\geq\C{n''}{6}.
\end{align}
Apply \eqref{eqn:I-2-1} and \eqref{eqn:I-2-2} to \eqref{eqn:I-2}, for any $n\geq \C{n''}{6}$, we have 
\begin{align}\label{eqn:I-3}
 &\int_{\mathbb{T}^2} \frac{1}{\rho^{4}} \int_{\gamma_1 \times \gamma_2} \mathbbm{1}_{\rho}(f^n_{\omega_1}(x), z) \mathbbm{1}_\rho(f^n_{\omega_2}(y), z) dm_{\gamma_1}(x) dm_{\gamma_2}(y) dm(z) \nonumber\\
 \leq&\frac{\C{C}{0}\pi}{\rho^2}\int_{\gamma_2} m_{\gamma_1}((f^n_{\omega_1})^{-1}(\{x_n\in \gamma_1^n: d(x_n,f^n_{\omega_2}(y))< 2\rho\}))dm_{\gamma_2}(y)\nonumber\\
 \leq&\frac{\C{C}{0}\pi}{\rho^2}\cdot\frac{4\C{C''}{0}\lambda_{u,-}^{-n}\rho}{\sin(\C{\theta}{0}/2)}\cdot m_{\gamma_2}((f^n_{\omega_2})^{-1}(\{y\in\gamma_2^n: d(y,\gamma_1^n)<2\rho\}))\nonumber\\
 \leq&\frac{\C{C}{0}\pi}{\rho^2}\cdot\frac{4\C{C''}{0}\lambda_{u,-}^{-n}\rho}{\sin(\C{\theta}{0}/2)}\cdot\frac{16\C{C''}{0}\lambda_{u,-}^{-n}\rho}{\C{C}{10}\lambda^n \sin(\C{\theta}{0}/2)}=\frac{64\C{C}{0}(\C{C''}{0})^2\lambda_{u,-}^{-2n}\pi}{\C{C}{10}\lambda^n\sin^2(\C{\theta}{0}/2)}=:\C{C'''}{11}(n).
\end{align}
Let $\C{C'}{11}(n,\C{\rho}{1}):=\C{C''}{11}(n,\C{\rho}{1})\C{C'''}{11}(n)$. \eqref{eqn:I.integrand.est} then follows from \eqref{eqn:I-1} and \eqref{eqn:I-3}. This finishes the proof.
\end{proof}

\subsection{Conclusion of the proof}

\begin{proof}[Proof of Theorem \ref{thm.mainthm}]
Let $\nu' = \nu_{\C{\rho}{1}}$ and fix $n'>\C{n}{6}(\varepsilon)$. In particular, for any $n\in\ZZ_{\geq0}$, we have $\mu^{*n}*\nu'(\TT^2)=\nu'(\TT^2)$. For simplicity, we write $\widehat{c}_n:=\C{C}{11}(n,\C{\rho}{1})$.

\begin{claim}\label{claim.upperbound}
There exists a constant $K>0$ such that for any $\rho$ satisfying \eqref{eqn:rho.cond}, we have $\displaystyle \limsup_{m \to +\infty} \left\|\frac{1}{m}\sum_{i=0}^{m-1} \mu^{*i}* \nu'\right\|^2_\rho \leq K.$ 
\end{claim}
\begin{proof}
In what follows, write $M_{\rho,n'} := \max \{\|\mu^{*r}*\nu_{\rho_0}\|^2_\rho: r=0, \cdots, n'-1\}$.  By Lemma \ref{lem.lasotayork}, we have
\begin{align*}
&\left\|\frac{1}{m} \sum_{i=0}^{m-1} \mu^{*i}*\nu'\right\|_\rho^2 
\leq \frac{1}{m}\sum_{r=0}^{n'-1} \sum_{l=0}^{\left[\frac{m}{n'}\right]} \left\|\mu^{*ln' + r}* \nu'\right\|^2_\rho \\
\leq& \frac{1}{m}\sum_{r=0}^{n'-1} \sum_{l=0}^{\left[\frac{m}{n'}\right]}\left( \widehat{\lambda}^{l n'} \left\|\mu^{*r}* \nu'\right\|^2_\rho + \widehat{c}_{n'}\left(\sum_{j=0}^l \widehat{\lambda}^{jn'} \right)(\nu'(\TT^2))^2 \right)  \\
\leq& \frac{1}{m}\sum_{r=0}^{n'-1} \sum_{l=0}^{\left[\frac{m}{n'}\right]}\left( \widehat{\lambda}^{l n'} \left\|\mu^{*r}* \nu'\right\|^2_\rho + \frac{\widehat{c}_{n'} (\nu'(\TT^2))^2}{1-\widehat{\lambda}^{n'}}\right)\\
\leq & \frac{1}{m}\sum_{r=0}^{n'-1} \sum_{l=0}^{\left[\frac{m}{n'}\right]}M_{\rho, n'} \widehat{\lambda}^{ln'} + \frac{1}{m}\sum_{r=0}^{n'-1} \sum_{l=0}^{\left[\frac{m}{n'}\right]} \frac{\widehat{c}_{n'} (\nu'(\TT^2))^2}{1-\widehat{\lambda}^{n'}}
\leq \frac{M_{\rho, n'} n' K'}{m} + \frac{\widehat{c}_{n'} (\nu'(\TT^2))^2}{1-\widehat{\lambda}^{n'}}.
\end{align*}
Therefore,
\[
\displaystyle \limsup_{m\to + \infty}  \left\|\frac{1}{m} \sum_{i=0}^{m-1} \mu^{*i}*\nu'\right\|_\rho^2 \leq \frac{\widehat{c}_{n'} (\nu'(\TT^2))^2}{1-\widehat{\lambda}^{n'}}=: K \qedhere
\]
\end{proof}

By Lemma \ref{lemma.srbconvergence}, the measure $\frac{1}{m{\nu'(\TT^2)}} \sum_{i=0}^{m-1} \mu^i_*\nu'$ converges to the unique SRB measure $\nu$.  By Lemma \ref{lem.normconvergence}, for any $\rho$
\[
\displaystyle \|\nu\|^2_\rho =\frac{1}{{(\nu'(\TT^2))^2}} \lim_{m\to +\infty} \left\|\frac{1}{m} \sum_{i=0}^{m-1} \mu^{*i}*\nu'\right\|^2_\rho \leq \frac{K}{{(\nu'(\TT^2))^2}}.
\]
Since this is true for any $\rho>0$ small enough,   we obtain
\[
\displaystyle \liminf_{\rho \to 0} \|\nu\|^2_\rho \leq \frac{K}{{(\nu'(\TT^2))^2}}.
\]
By Lemma \ref{lem.liminfnorm},  $\nu$ is absolutely continuous with respect to the smooth measure $m$ on $\mathbb{T}^2$ and $\lim_{\rho \to 0} \|\nu\|_{\rho} = \left\|\frac{d\nu}{d m}\right\|_{\mathrm{L}^2(m)}$. \qedhere

\end{proof}

\subsection*{The proofs of Corollaries \ref{thm.thmforsurfaces} and \ref{thm.invariantmeasure}}

In this section we will suppose that $f,g\in \mathrm{Diff}_m^2(\mathbb{T}^2)$ verify conditions \textbf{(C1)-(C4)}. Fix $\beta \in (0, \frac{1}{2}]$ and let $\mathcal{U}_f$ and $\mathcal{U}_g$ be the open sets given by Theorem \ref{thm.mainthm}.

\begin{proof}[Proof of Corollary \ref{thm.thmforsurfaces}]
Let $\mu$ be a probability measure verifying the hypothesis of Corollary \ref{thm.thmforsurfaces}. Conditions \textbf{(C1),(C2)} and \textbf{(C4)} allows us to apply the main result in \cite{Brown-Hertz} to conclude that any $\mu$-stationary ergodic measure $\nu$ is either SRB or atomic (condition \textbf{(C4)} implies that the stable direction is random).  The conclusion is a direct consequence of Theorem \ref{thm.mainthm}. \qedhere

\end{proof}

\begin{proof}[Proof of Corollary \ref{thm.invariantmeasure}]
Fix $\widehat{f} \in \mathcal{U}_f$ and $\widehat{g} \in \mathcal{U}_g$, and suppose that $\nu$ is a non-atomic invariant measure for $\widehat{f}$ and $\widehat{g}$. Consider $\mu = \frac{1}{2} \delta_{\widehat{f}} + \frac{1}{2} \delta_{\widehat{g}}$.  Clearly $\mu$ verify the hypothesis of Corollary \ref{thm.thmforsurfaces}. Since $\nu$ is invariant by the two diffeomorphisms, we have that $\nu$ is $\mu$-stationary.  Let $\Gamma$ be the semigroup generated by $\widehat{f}$ and $\widehat{g}$. Observe that ergodic, atomic $\mu$-stationary measures are supported on points with finite $\Gamma$-orbits. In particular, there are at most countably many of them.  Therefore, we conclude that there are at most countably many different ergodic $\mu$-stationary measures.  If $\nu$ were not $\mu$-ergodic, there would be a finite orbit with positive $\nu$-measure. This is not possible since $\nu$ is non-atomic. Therefore, $\nu$ is $\mu$-ergodic. The conclusion follows directly from Corollary \ref{thm.thmforsurfaces}, using that $\nu$ is non-atomic.
\end{proof}

\section{Equidistribution and orbit closure classification} \label{sec:equidistribution}

In this section, we prove Theorem \ref{thm.equidistribution},  and Corollaries \ref{thm.orbitclosureclassification} and \ref{thm.genericminimality}.  The proof of Theorem \ref{thm.equidistribution} and Corollary \ref{thm.orbitclosureclassification} is essentially the same as the proofs of Propositions $4.1$ and $4.2$ from \cite{chung}. 
In this section, let $\mathcal{U}_f$ and $\mathcal{U}_g$ be $C^2$-neighborhoods of $f$ and $g$, respectively, such that conditions (C1)-(C4) hold for any pair $(\widehat{f}, \widehat{g}) \in \mathcal{U}_f \times \mathcal{U}_g$.

\begin{dfn}
A probability measure $\mu$ on $\text{Diff}^2(\mathbb{T}^2)$ is \emph{uniformly expanding} if there are constants $C>0$ and $N\in \mathbb{N}$ such that for every $x\in \mathbb{T}^2$ and unit vector $v\in T_x \mathbb{T}^2$, it holds
\[
\displaystyle \int \log \|Df^n_\omega(x) v\| d\mu^N(\omega) > C.
\]
\end{dfn}

In other words, one sees uniform expansion at a uniform time on average for every point and direction.

\begin{lem}\label{lemma.uniformexpansionforus}

Let $\mu$ be a probability measure supported on $\mathcal{U}_f \cup \mathcal{U}_g$ such that $\mu(\mathcal{U}_f)>0$ and $\mu(\mathcal{U}_g)>0$. Then $\mu$ is uniformly expanding. 
\end{lem}

\begin{proof}
Since the stable distribution is not invariant for $\mu$-almost every $h$, the Lemma follows as a direct application of Proposition $3.17$ from \cite{chung}.
\end{proof}

Let $S$ be a finite set contained in $\mathcal{U}_f \cup \mathcal{U}_g$ such that $S$ intersects both $\mathcal{U}_f $ and $\mathcal{U}_g$, and let $\Gamma_S$ be the semigroup generated by $S$.  Let $\mu$ be as in the statement of Theorem \ref{thm.equidistribution}. By Lemma \ref{lemma.uniformexpansionforus}, $\mu$ is uniform expanding.  Below, we will state several results from \cite{chung} that hold in more generality, under some integrability condition.

\begin{prop}[Proposition $4.6$ from \cite{chung}]\label{proposition.countablefiniteorbits}
The number of points with finite $\Gamma_S$-orbit is countable. 
\end{prop}

%
%

\begin{lem}[Lemma $4.7$ in \cite{chung}]\label{lemma.omegaset}
Let $\mathcal{N}$ be a finite $\Gamma_S$-orbit in $\mathbb{T}^2$. For any $\varepsilon >0$, there exists an open set $\Omega_{\mathcal{N},\varepsilon}$ containing $\mathcal{N}$, such that for any compact set $F \subset \mathbb{T}^2 / \mathcal{N}$, there exists a positive integer $n_F$, such that for all $x\in F$, and $n> n_F$, we have
\[
\displaystyle \left(\frac{1}{n} \sum_{i=0}^{n-1} \mu^{*i}* \delta_x\right)(\Omega_{\mathcal{N},\varepsilon}) < \varepsilon. 
\]
\end{lem}

The proof of Proposition \ref{proposition.countablefiniteorbits} and Lemma \ref{lemma.omegaset} uses a Margulis function (see Lemma $4.3$ from \cite{chung}). 

\begin{proof}[Proof of Theorem \ref{thm.equidistribution}]
The proof is exactly the same as the proof of Proposition $4.1$ from \cite{chung}, where the unique $\mu$-stationary SRB measure $\nu$ takes the role of the smooth measure $m$ in the proof.  The main property of $m$ used by Chung is that it is fully supported. Observe that \Cref{lem.fullsupport} gives us that the unique SRB $\mu$-stationary measure is fully supported.
\end{proof}

\begin{proof}[Proof of Theorem \ref{thm.orbitclosureclassification}]
Let $\nu$ be the unique $\mu$-stationary SRB measure. By Lemma \ref{lem.fullsupport}, $\nu$ is fully supported, in particular, it gives positive measure to any open set. The proof then follows by Theorem \ref{thm.equidistribution}.
\end{proof}

\begin{proof}[Proof of Theorem \ref{thm.genericminimality}]
Fix $\widehat{g} \in \mathcal{U}_g$.  For each $n\in \mathbb{N}$, the set of periodic points of period $n$, $\text{Per}(\widehat{g})$,  is finite.  It is easy to see that the set $\mathcal{U}_{f,\widehat{g},n}:= \{\widehat{f} \in \mathcal{U}_f: \text{Per}_n(\widehat{f}) \cap \text{Per}_n(\widehat{g}) = \emptyset\}$ is open.  It is also easy to see that $\mathcal{U}_{f,\widehat{g}, n}$ is dense.  By Baire's theorem, the set \
\[
\displaystyle \mathcal{R}_{\widehat{g}}:=\bigcap_{n\in \mathbb{N}} \mathcal{U}_{f,\widehat{g},n}
\]
 is a dense $\text{G}_\delta$ subset of $\mathcal{U}_f$.  Let $\widehat{f} \in \mathcal{R}_{\widehat{g}}$ and let $S = \{\widehat{f}, \widehat{g}\}$. Since $\text{Per}(\widehat{f}) \cap \text{Per}(\widehat{g}) = \emptyset$, there are no finite $\Gamma_S$-orbit. By Theorem \ref{thm.orbitclosureclassification}, every $\Gamma_S$-orbit is dense and the action is minimal. 
\end{proof}

\bibliographystyle{alpha}
\bibliography{bibliography}

\end{document}